\documentclass[12pt,reqno]{amsart}
\usepackage[a4paper, margin=1.25in]{geometry}
\makeatletter
\@namedef{subjclassname@2020}{%
  \textup{2020} Mathematics Subject Classification}
\makeatother
\title[Campana points]{Campana points and powerful values of norm forms}
\author{Sam Streeter}
\address{School of Mathematics, Fry Building, Woodland Road, Bristol, BS8 1UG}
\email{sam.streeter@bristol.ac.uk}
\usepackage{amsmath}
\usepackage{amssymb}
\usepackage{amsthm}
\usepackage{enumitem}
\usepackage{tikz-cd}
\usepackage{soul}
\usepackage[linktocpage=true]{hyperref}
\usepackage{mathrsfs}
\theoremstyle{definition}
\newtheorem{mydef}{Definition}
\newtheorem{note}[mydef]{Note}
\newtheorem{remark}[mydef]{Remark}

\theoremstyle{plain}
\newtheorem{theorem}[mydef]{Theorem}
\newtheorem{proposition}[mydef]{Proposition}
\newtheorem{lemma}[mydef]{Lemma}
\newtheorem{corollary}[mydef]{Corollary}
\newtheorem{conjecture}[mydef]{Conjecture}
\numberwithin{mydef}{section}
\numberwithin{equation}{section}
\let\originalleft\left
\let\originalright\right
\renewcommand{\left}{\mathopen{}\mathclose\bgroup\originalleft}
\renewcommand{\right}{\aftergroup\egroup\originalright}
\renewcommand{\Re}{\operatorname{Re}}

\DeclareMathOperator{\ab}{ab}

\DeclareMathOperator{\Cl}{Cl}

\DeclareMathOperator{\degree}{deg}

\DeclareMathOperator{\Gal}{Gal}

\DeclareMathOperator{\Hom}{Hom}

\DeclareMathOperator{\om}{\boldsymbol{\omega}}

\DeclareMathOperator{\Pic}{Pic}
\DeclareMathOperator{\Proj}{Proj}
\DeclareMathOperator{\rank}{rank}
\DeclareMathOperator{\Res}{Res}
\DeclareMathOperator{\Spec}{Spec}

\DeclareMathOperator{\Val}{Val}
\DeclareMathOperator{\vol}{vol}
    \DeclareFontFamily{U}{wncy}{}
    \DeclareFontShape{U}{wncy}{m}{n}{<->wncyr10}{}
    \DeclareSymbolFont{mcy}{U}{wncy}{m}{n}
    \DeclareMathSymbol{\Sh}{\mathord}{mcy}{"58} 
\begin{document}
%%%%%%%%%%%%%%%%%%%%%%%%%%%%%%%%%%%%%%%%%%%%%%%%%%%%%%%%%%%%%%%%%%%%%%%%%%%%%%%%%%%%%%%%%%%%%%%%%%%%%%%%%%%%%%%%%%%%%%%%%%%%%%%%
\begin{abstract}
We give an asymptotic formula for the number of weak Campana points of bounded height on a family of orbifolds associated to norm forms for Galois extensions of number fields. From this formula we derive an asymptotic for the number of elements with $m$-full norm over a given Galois extension of $\mathbb{Q}$. We also provide an asymptotic for Campana points on these orbifolds which illustrates the vast difference between the two notions, and we compare this to the Manin-type conjecture of Pieropan, Smeets, Tanimoto and V\'arilly-Alvarado.
\end{abstract}
\maketitle
\setcounter{tocdepth}{1}
\tableofcontents
%%%%%%%%%%%%%%%%%%%%%%%%%%%%%%%%%%%%%%%%%%%%%%%%%%%%%%%%%%%%%%%%%%%%%%%%%%%%%%%%%%%%%%%%%%%%%%%%%%%%%%%%%%%%%%%%%%%%%%%%%%%%%%%%
\section{Introduction} \label{sectionI}
%%%%%%%%%%%%%%%%%%%%%%%%%%%%%%%%%%%%%%%%%%%%%%%%%%%%%%%%%%%%%%%%
The theory of Campana points is of growing interest in arithmetic geometry due to its ability to interpolate between rational and integral points. Two competing notions of Campana points can be found in the literature, both extending a definition of ``orbifold rational points'' for curves within Campana's theory of ``orbifoldes g\'eom\'etriques'' in \cite{OSV}, \cite{FMS}, \cite{OG} and \cite{SMAHA}. They capture the idea of rational points which are \emph{integral with respect to a weighted boundary divisor}. These two notions have been termed \emph{Campana points} and \emph{weak Campana points} in the recent paper \cite{CPBH} of Pieropan, Smeets, Tanimoto and V\'arilly-Alvarado, in which the authors initiate a systematic quantitative study of points of the former type on smooth Campana orbifolds and prove a logarithmic version of Manin's conjecture for Campana points on vector group compactifications. The only other quantitative results in the literature are found in \cite{STSN}, \cite{SNH}, \cite{AHDO}, \cite{HMTV} and \cite{CPBC}, and the former four of these indicate the close relationship between Campana points and $m$-full solutions of equations. We recall that, given $m \in \mathbb{Z}_{\geq 2}$, we say that $n \in \mathbb{Z} \setminus \{0\}$ is \emph{$m$-full} if all primes in the prime decomposition of $n$ have multiplicity at least $m$.

In this paper, we bring together the perspectives in the above papers and provide the first result for Campana points on singular orbifolds. As observed in \cite[\S1.1]{CPBH}, the study of weak Campana points of bounded height is challenging and requires new ideas for the regularisation of certain Fourier transforms, and these ideas for the orbifolds in consideration are the main innovation of this paper. We adopt a height zeta function approach, similar to the one employed in \cite{CPBH} and modelled on the work of Loughran in \cite{NVF} and Batyrev and Tschinkel in \cite{RPBH} on toric varieties, in order to prove log Manin conjecture-type results for both types of Campana points on $\left(\mathbb{P}^{d-1}_K,\left(1-\frac{1}{m}\right)Z\left(N_{\om}\right)\right)$, where $Z\left(N_{\om}\right)$ is the zero locus of a norm form $N_{\om}$ associated to a $K$-basis $\om$ of a Galois extension of number fields $E/K$ of degree $d \geq 2$ coprime to $m \in \mathbb{Z}_{\geq 2}$ if $d$ is not prime. Although toric varieties also play a prominent role in \cite{HMTV}, the tori there are split, whereas we shall work with anisotropic tori. When $K = \mathbb{Q}$, we derive from the result for weak Campana points an asymptotic for the number of elements of $E$ of bounded height with $m$-full norm over $\mathbb{Q}$. We compare the result for Campana points to a conjecture of Pieropan, Smeets, Tanimoto and V\'arilly-Alvarado \cite[Conj.~1.1,~p.~3]{CPBH}.
%%%%%%%%%%%%%%%%%%%%%%%%%%%%%%%%%%%%%%%%%%%%%%%%%%%%%%%%%%%%%%%%
\subsection{Results} \label{sectionR}
\begin{theorem} \label{mainthm}
Let $E/K$ be a Galois extension of number fields of degree $d \geq 2$, and let $m \geq2$ be an integer which is coprime to $d$ if $d$ is not prime. Let $\om$ be a $K$-basis of $E$. Denote by $\Delta^{\om}_m$ the $\mathbb{Q}$-divisor $\left(1-\frac{1}{m}\right)Z\left(N_{\om}\right)$ of $\mathbb{P}^{d-1}_K$ for $N_{\om}$ the norm form corresponding to $\om$. Let $H$ denote the anticanonical height function on $\mathbb{P}^{d-1}_K$ from Definition \ref{bthtdef}. Then there exists an explicit finite set $S\left(\om\right) \subset \Val\left(K\right)$ such that, for any finite set of places $S \supset S\left(\om\right)$, the number $N\left(\left(\mathbb{P}^{d-1}_{K},\Delta^{\om}_m\right),H,B,S\right)$ of weak Campana $\mathcal{O}_{K,S}$-points of height at most $B \in \mathbb{R}_{\geq 1}$ on the orbifold $\left(\mathbb{P}^{d-1}_K,\Delta^{\om}_m\right)$ with respect to the model $\mathbb{P}^{d-1}_{\mathcal{O}_{K,S}}$ of $\mathbb{P}^{d-1}_K$ has the asymptotic formula
$$
N\left(\left(\mathbb{P}^{d-1}_K,\Delta^{\om}_m\right),H,B,S\right) \sim c\left(\om,m,S\right)B^{\frac{1}{m}}\left(\log B\right)^{b\left(d,m\right)-1}
$$
as $B \rightarrow \infty$ for some explicit positive constant $c\left(\om,m,S\right)$, where
$$
b\left(d,m\right) = \frac{1}{d}\left({{d+m-1}\choose{d-1}} - {{m-1}\choose{d-1}}\right).
$$
\end{theorem}
\begin{note}
If $\om$ is a relative integral basis of $E/K$, then $S\left(\om\right) = S_{\infty}$, the set of archimedean places of $K$, in Theorem \ref{mainthm} (see Remark \ref{ribrem}).
\end{note}
Each rational point $P \in \mathbb{P}^{d-1}\left(\mathbb{Q}\right)$ possesses precisely two sets of coordinates in $\mathbb{Z}^d_{\textrm{prim}} = \{\left(x_0,\dots,x_{d-1}\right) \in \mathbb{Z}^d : \gcd\left(x_0,\dots,x_{d-1}\right) = 1\}$. Interpreting $H$ and $N_{\om}$ as functions on this set, we immediately obtain the following result.
\begin{corollary}
Taking $K=\mathbb{Q}$ and letting $\om$ be an integral basis with the notation and hypotheses of Theorem \ref{mainthm}, we have
$$
\#\{x \in \mathbb{Z}^d_{\textrm{prim}} : H\left(x\right) \leq B, \, \textrm{$N_{\om}\left(x\right)$ is $m$-full}\} \sim 2c\left(\om,m,S_{\infty}\right) B^{\frac{1}{m}} \left(\log B\right)^{b\left(d,m\right)-1}.
$$
\end{corollary}
Arithmetically special (e.g.\ prime, square-free) values of norm forms are a topic of long-standing interest in number theory (see e.g.\ \cite{PF}, \cite{SFRF}).

Campana points are only defined and studied for smooth orbifolds (i.e.\ smooth varieties for which the orbifold divisor has strict normal crossings support) in \cite{CPBH}. In order to study the Campana points of $\left(\mathbb{P}^{d-1}_K,\Delta^{\om}_m\right)$, which is smooth only when $d = 2$, we must first generalise the definition of Campana points, which we do in Section \ref{subsectionCP}.
Using the same strategy employed in the proof of Theorem \ref{mainthm}, we then derive an asymptotic for the number of Campana points on $\left(\mathbb{P}^{d-1}_K,\Delta^{\om}_m\right)$.
\begin{theorem} \label{campthm}
With the notation and hypotheses of Theorem \ref{mainthm}, denote by $\widetilde{N}\left(\left(\mathbb{P}^{d-1}_K,\Delta^{\om}_m\right),H,B,S\right)$ the number of Campana $\mathcal{O}_{K,S}$-points on $\left(\mathbb{P}^{d-1}_K,\Delta^{\om}_m\right)$ of height at most $B \in \mathbb{R}_{\geq 1}$ with respect to $H$. Then there exists an explicit positive constant $\widetilde{c}\left(\om,m,S\right)$ such that, as $B \rightarrow \infty$, we have
$$
\widetilde{N}\left(\left(\mathbb{P}^{d-1}_K,\Delta^{\om}_m\right),H,B,S\right) \sim \widetilde{c}\left(\om,m,S\right) B^{\frac{1}{m}}.
$$
\end{theorem}
\begin{remark}
It is not clear if the exponent of the logarithm in Theorem \ref{mainthm} admits a geometric interpretation as it does in Theorem \ref{campthm} (cf.\ \cite[Conj.~1.1,~p.~3]{CPBH}).
\end{remark}
%%
%%%%%%%%%%%%%%%%%%%%%%%%%%%%%%%%%%%%%%%%%%%%%%%%%%%%%%%%%%%%%%%%
\subsection*{Acknowledgements} \label{subsectionA}
Parts of this work were completed during the program ``Reinventing Rational Points'' at the Institut Henri Poincar\'e and the conference ``Topics in Rational and Integral Points'' at the Universit\"at Basel; the author thanks the organisers of both events for their hospitality. Thanks go to the authors of \cite{CPBH} for their helpful feedback; to Daniel Loughran for his guidance as an expert and mentor; to Gregory Sankaran and David Bourqui for improvements arising from their examination of this work within the author's PhD thesis; to the anonymous referee for their thoughtful comments and suggestions, and to the Heilbronn Institute for Mathematical Research for their support.
%%%%%%%%%%%%%%%%%%%%%%%%%%%%%%%%%%%%%%%%%%%%%%%%%%%%%%%%%%%%%%%%
\subsection*{Conventions} \label{subsectionC}
\subsubsection*{Algebra}
We take $\mathbb{N} = \mathbb{Z}_{\geq 1}$. We denote by $R^*$ the group of units of a ring $R$. Given a group $G$, we denote by $1_G$ the identity element of $G$, and for any $n \in \mathbb{N}$, we set $G[n] = \{g \in G: g^n = 1_G\}$. For any perfect field $F$, we fix an algebraic closure $\overline{F}$ and set $G_F = \Gal\left(\overline{F}/F\right)$. Given a topological group $G$, we denote by $G^{\wedge} = \Hom\left(G,S^1\right)$ its group of continuous characters, where $S^1 = \{z \in \mathbb{C}: z \overline{z} = 1\} \subset \mathbb{C}^*$ is the circle group. A monomial in the variables $x_1,\dots,x_n$ is a product $x_1^{a_1}\dotsm x_n^{a_n}$, $\left(a_1,\dots,a_n\right) \in \mathbb{Z}_{\geq 0}^n$. For any $n \in \mathbb{N}$, we denote by $\mu_n$ the group of $n$th roots of unity and by $S_n$ the symmetric group of order $n$.
\subsubsection*{Geometry}
We denote by $\mathbb{P}_R^n$ the projective $n$-space over the ring $R$. We omit the subscript if the ring $R$ is clear. Given a homogeneous polynomial $f \in R[x_0,\dots,x_n]$, we denote by $Z\left(f\right) =  \Proj R[x_0,\dots,x_n]/\left(f\right)$ the zero locus of $f$ viewed as a closed subscheme of $\mathbb{P}_R^n$.
A variety over a field $F$ is a geometrically integral separated scheme of finite type over $F$. Given a variety $X$ defined over $F$ and an extension $E/F$, we denote by $X_E= V \times_{\Spec F} \Spec E$ the base change of $X$ over $E$, and we write $\overline{X} = X \times_{\Spec F} \Spec \overline{F}$. When $F = K$ and $E = K_v$ for a number field $K$ and a place $v$ of $K$, we write $X_{v} = X_{K_v}$. Given a field $F$, we define $\mathbb{G}_{m,F} = \Spec F[x_0,x_1]/\left(x_0 x_1-1\right)$. We omit the subscript $F$ if the field is clear.
\subsubsection*{Number theory}
Given an extension of number fields $L/K$ with $K$-basis $\om = \{\omega_0,\dots,\omega_{d-1}\}$, we write $N_{\om}\left(x_0,\dots,x_{d-1}\right) = N_{L/K}\left(x_0 \omega_0 + \dots + x_{d-1}\omega_{d-1}\right)$ for the associated norm form. We denote by $\Val\left(K\right)$ the set of valuations of a number field $K$, and we denote by $S_{\infty}$ the set of archimedean valuations. For $v \in \Val\left(K\right)$, we denote by $\mathcal{O}_v$ the maximal compact subgroup of $K_v$. For a finite set of places $S$ containing $S_{\infty}$, we denote by $\mathcal{O}_{K,S} = \{\alpha \in K : \alpha \in \mathcal{O}_v \textrm{ for all $v \not\in S$}\}$ the ring of algebraic $S$-integers of $K$. We write $\mathcal{O}_K = \mathcal{O}_{K,S_{\infty}}$. For $v \in \Val \left(K\right)$ non-archimedean, we denote by $\pi_v$ and $q_v$ a uniformiser for the residue field of $K_v$ and the size of the residue field of $K_v$ respectively. If $v \mid \infty$, then we set $\log q_v = 1$. For each $v \in \Val \left(K\right)$, we choose the absolute value $|x|_v = |N_{K_v/\mathbb{Q}_p}\left(x\right)|_p$ for the unique $p \in \Val \left(\mathbb{Q}\right)$ with $v \mid p$ and the usual absolute value $|\cdot|_p$ on $\mathbb{Q}_p$. We denote by $\mathbb{A}_K =\widehat{\prod}_{v \in \Val \left(K\right)}^{\mathcal{O}_v}K_v$ the adele ring of $K$ with the restricted product topology.
%%%%%%%%%%%%%%%%%%%%%%%%%%%%%%%%%%%%%%%%%%%%%%%%%%%%%%%%%%%%%%%%%%%%%%%%%%%%%%%%%%%%%%%%%%%%%%%%%%%%%%%%%%%%%%%%%%%%%%%%%%%%%%%%
\section{Background} \label{sectionB}
%%%%%%%%%%%%%%%%%%%%%%%%%%%%%%%%%%%%%%%%%%%%%%%%%%%%%%%%%%%%%%%%
\subsection{Campana points} \label{subsectionCP}
In this section we define Campana orbifolds, Campana points and weak Campana points, generalising the definitions in \cite[\S3.2]{CPBH} in such a way that the exponents in Theorem \ref{campthm} match those in \cite[Conj.~1.1,~p.~3]{CPBH}.
\begin{mydef}
A \emph{Campana orbifold} over a field $F$ is a pair $\left(X,D_{\epsilon}\right)$ consisting of a proper, normal variety $X$ over $F$ and an effective Cartier $\mathbb{Q}$-divisor
$$
D_{\epsilon} = \sum_{\alpha \in \mathcal{A}} \epsilon_{\alpha} D_{\alpha}
$$
on $X$, where the $D_{\alpha}$ are prime divisors and $\epsilon_{\alpha} = 1 - \frac{1}{m_{\alpha}}$ for some $m_{\alpha} \in \mathbb{Z}_{\geq 2} \cup \{\infty\}$ (by convention, we take $\frac{1}{\infty} = 0$).
We define the \emph{support} of the $\mathbb{Q}$-divisor $D_{\epsilon}$ to be
$$
D_{\textrm{red}} = \sum_{\alpha \in \mathcal{A}} D_{\alpha}.
$$
We say that $\left(X,D_{\epsilon}\right)$ is \emph{smooth} if $X$ is smooth and $D_{\textrm{red}}$ has strict normal crossings (see \cite[\S41.21]{SP} for the definition of strict normal crossings divisors).
\end{mydef}
Let $\left(X,D_{\epsilon}\right)$ be a Campana orbifold over a number field $K$. Let $S \subset \Val\left(K\right)$ be a finite set containing $S_{\infty}$.
\begin{mydef}
A \emph{model} of $\left(X,D_{\epsilon}\right)$ over $\mathcal{O}_{K,S}$ is a pair $\left(\mathcal{X},\mathcal{D}_{\epsilon}\right)$, where $\mathcal{X}$ is a flat proper model of $X$ over $\mathcal{O}_{K,S}$ (i.e.\ a flat proper $\mathcal{O}_{K,S}$-scheme with $\mathcal{X}_{\left(0\right)} \cong X$) and $\mathcal{D}_{\epsilon} = \sum_{\alpha \in \mathcal{A}}\epsilon_{\alpha} \mathcal{D}_{\alpha}$ for $\mathcal{D}_{\alpha}$ the Zariski closure of $D_{\alpha}$ in $\mathcal{X}$.

Define $\mathcal{D}_{\textrm{red}} = \sum_{\alpha \in \mathcal{A}}\mathcal{D}_{\alpha}$. Denote by $\mathcal{D}_{\alpha_v}$, $\alpha_v \in \mathcal{A}_v$ the irreducible components of $\mathcal{D}_{\textrm{red}}$ over $\Spec \mathcal{O}_v$. We write $\alpha_v \mid \alpha$ if $\mathcal{D}_{\alpha_v} \subset \mathcal{D}_{\alpha}$.
\end{mydef}
Let $\left(\mathcal{X},\mathcal{D}_{\epsilon}\right)$ be a model for $\left(X,D_{\epsilon}\right)$ over $\mathcal{O}_{K,S}$. For $v \not\in S$, any $P \in X\left(K\right)$ induces some $\mathcal{P}_v \in \mathcal{X}\left(\mathcal{O}_v\right)$ by the valuative criterion of properness \cite[Thm.~II.4.7,~p.~101]{AG}.
\begin{mydef}
Let $P \in X\left(K\right)$ and take a place $v \not\in S$. For each $\alpha_v \in \mathcal{A}_v$, we define the \emph{local intersection multiplicity} $n_v\left(\mathcal{D}_{\alpha_v},P\right)$ of $\mathcal{D}_{\alpha_v}$ and $P$ at $v$ to be $\infty$ if $\mathcal{P}_v \subset \mathcal{D}_{\alpha_v}$, and the colength of the ideal $\mathcal{P}_v^*\mathcal{D}_{\alpha_v} \subset \mathcal{O}_v$ otherwise. We then define the quantities
$$
n_v\left(\mathcal{D}_{\alpha},P\right) = \sum_{\alpha_v \mid \alpha} n_v\left(\mathcal{D}_{\alpha_v},P\right), \quad n_v\left(\mathcal{D}_{\epsilon},P\right) = \sum_{\alpha \in \mathcal{A}} \epsilon_{\alpha} n_{v}\left(\mathcal{D}_{\alpha},P\right).
$$
\end{mydef}
We are now ready to define weak Campana points and Campana points. Both notions arise from \cite{BGNT}, with the former appearing in its current form in \cite[\S1]{CPVC}.
\begin{mydef}
We say that $P \in X\left(K\right)$ is a \emph{weak Campana $\mathcal{O}_{K,S}$-point} of $\left(\mathcal{X},\mathcal{D}_{\epsilon}\right)$ if the following implications hold for all places $v \not\in S$ of $K$ and for all $\alpha \in \mathcal{A}$:
\begin{enumerate}
\item If $\epsilon_\alpha = 1$ (meaning $m_{\alpha} = \infty$), then $n_v\left(\mathcal{D}_{\alpha},P\right) = 0$.
\item If $n_v\left(\mathcal{D}_{\epsilon},P\right) > 0$, then
$$
\sum_{\alpha \in \mathcal{A}} \frac{1}{m_{\alpha}}n_v\left(\mathcal{D}_{\alpha},P\right) \geq 1.
$$
\end{enumerate}
We denote the set of weak Campana $\mathcal{O}_{K,S}$-points of $\left(\mathcal{X},\mathcal{D}_{\epsilon}\right)$ by $\left(\mathcal{X},\mathcal{D}_{\epsilon}\right)_{\mathbf{w}}\left(\mathcal{O}_{K,S}\right)$.
\end{mydef}
\begin{mydef}
We say that $P \in X\left(K\right)$ is a \emph{Campana $\mathcal{O}_{K,S}$-point} of $\left(\mathcal{X},\mathcal{D}_{\epsilon}\right)$ if the following implications hold for all places $v \not\in S$ of $K$ and for all $\alpha \in \mathcal{A}$.
\begin{enumerate}[label=(\roman*)]
\item If $\epsilon_{\alpha} = 1$ (meaning $m_{\alpha} = \infty$), then $n_v\left(\mathcal{D}_{\alpha},P\right) = 0$. 
\item If $\epsilon_\alpha \neq 1$ and $n_v\left(\mathcal{D}_{\alpha_v},P\right) > 0$ for some $\alpha_v \mid \alpha$, then
$
n_v\left(\mathcal{D}_{\alpha_v},P\right) \geq m_{\alpha}.
$
\end{enumerate}
We denote the set of Campana $\mathcal{O}_{K,S}$-points of $\left(\mathcal{X},\mathcal{D}_{\epsilon}\right)$ by $\left(\mathcal{X},\mathcal{D}_{\epsilon}\right)\left(\mathcal{O}_{K,S}\right)$.
\end{mydef}
\begin{remark}
Informally, weak Campana points are rational points $P \in X\left(K\right)$ avoiding $\cup_{\epsilon_\alpha = 1}\mathcal{D}_{\alpha}$ which, upon reduction modulo any place $v \not\in S$, either do not lie on $D_{\textrm{red}}$ or lie on $D_\alpha$ with multiplicity at least $m_\alpha$ on average over $\alpha$. Similarly, Campana points are rational points $P \in X\left(K\right)$ avoiding $\cup_{\epsilon_\alpha = 1}\mathcal{D}_{\alpha}$ which, upon reduction modulo any place $v \not\in S$, either do not lie on $D_{\textrm{red}}$ or lie on each $v$-adic irreducible component of each $D_{\alpha}$ with multiplicity either $0$ or at least $m_{\alpha}$.
\end{remark}
\begin{remark}
Our definition of Campana points differs from the one in \cite[\S3.2]{CPBH}, in which one requires simply that $n_v\left(\mathcal{D}_\alpha,P\right) \geq m_{\alpha}$ instead of $n_v\left(\mathcal{D}_{\alpha_v},P\right) \geq m_{\alpha}$ in the second implication. If one were to apply this definition to the orbifold $\left(\mathbb{P}^{d-1}_K,\Delta^{\om}_m\right)$ of Theorem \ref{mainthm}, which is singular for all $d \geq 3$ as $Z\left(N_{\om}\right)$ is not a strict normal crossings divisor in this case, then the weak Campana points and the Campana points would be the same, but the asymptotic of Theorem \ref{mainthm} differs to \cite[Conj.~1.1,~p.~3]{CPBH} for $d \geq 3$ (at least if one takes the thin set there to be the empty set). Using the definitions above, we obtain the asymptotic for Campana points in Theorem \ref{campthm}, whose exponents match this conjecture.
\end{remark}
\begin{lemma} \label{genlem}
Let $\left(X,D_{\epsilon}\right)$ be a smooth Campana orbifold over a number field $K$ which is Kawamata log terminal (i.e.\ $\epsilon_{\alpha} < 1$ for all $\alpha \in \mathcal{A}$), and let $\left(\mathcal{X},\mathcal{D}_{\epsilon}\right)$ be a model of $\left(X,D_{\epsilon}\right)$ over $\mathcal{O}_{K,S}$ with $\mathcal{X}$ smooth over $\mathcal{O}_{K,S}$ and $\mathcal{D}_{\textrm{red}}$ a relative strict normal crossings divisor in $\mathcal{X}/\mathcal{O}_{K,S}$ as defined in \cite[\S2]{GMT}. Then the definition of Campana points on $\left(\mathcal{X},\mathcal{D}_{\epsilon}\right)$ above coincides with the one in \cite[\S3.2]{CPBH}.
\end{lemma}
\begin{proof}
Since $\mathcal{D}_{\textrm{red}}$ is a relative strict normal crossings divisor, each irreducible component $\mathcal{D}_{\alpha}$ is smooth over $\mathcal{O}_{K,S}$. In particular, its base change over $\Spec \mathcal{O}_v$ is smooth for any $v \not\in S$, so the divisors $\mathcal{D}_{\alpha_v}$, $\alpha_v \mid \alpha$ are disjoint. Then, for any rational point $P \in X\left(K\right)$, the reduction of $P$ at the place $v$ can lie on at most one of the divisors $\mathcal{D}_{\alpha_v}$, $\alpha_v \mid \alpha$, so $n_v\left(\mathcal{D}_{\alpha},P\right) = \sum_{\alpha_v \mid \alpha} n_v\left(\mathcal{D}_{\alpha_v},P\right)$ is either $0$ or at least $m_{\alpha}$ if and only if each $n_v\left(\mathcal{D}_{\alpha_v},P\right)$, $\alpha_v \mid \alpha$, is either $0$ or at least $m_{\alpha}$.
\end{proof}
%%
%%%%%%%%%%%%%%%%%%%%%%%%%%%%%%%%%%%%%%%%%%%%%%%%%%%%%%%%%%%%%%%%
\subsection{Toric varieties}
\begin{mydef}
An \emph{(algebraic) torus} over a field $F$ is an algebraic group $T$ over $F$ such that $\overline{T} \cong \mathbb{G}_m^n$ for some $n \in \mathbb{N}$. The \emph{splitting field} of a torus $T$ over a field $F$ is defined to be the smallest Galois field extension $E$ of $F$ for which $T_E \cong \mathbb{G}_m^n$.
\end{mydef}
\begin{mydef}
A \emph{toric variety} is a smooth projective variety $X$ equipped with a faithful action of an algebraic torus $T$ such that there is an open dense orbit containing a rational point.  
\end{mydef}
\begin{mydef}
Let $T$ be a torus over a field $F$. The \emph{character group} of $T$ is $X^*\left(\overline{T}\right) = \Hom\left(\overline{T},\mathbb{G}_m\right)$, and we have $X^*\left(T\right) = X^*\left(\overline{T}\right)^{G_F}$. The \emph{cocharacter group} of $T$ is $X_*\left(\overline{T}\right) = \Hom\left(X^*\left(\overline{T}\right), \mathbb{Z}\right)$, and we have $X_*\left(T\right) = X_*\left(\overline{T}\right)^{G_F}$. We let $X^*\left(T\right)_{\mathbb{R}} = X^*\left(T\right) \otimes_{\mathbb{Z}} \mathbb{R}$ and $X_*\left(T\right)_{\mathbb{R}} = X_*\left(T\right) \otimes_{\mathbb{Z}} \mathbb{R}$.
\end{mydef}
\begin{mydef}
An algebraic torus $T$ over a field $F$ is \emph{anisotropic} if it has trivial character group over $F$, i.e.\ $X^*\left(T\right) = 0.$
\end{mydef}
Let $T$ be a torus over a number field $K$ with splitting field $E$. Set $T_\infty = \prod_{v \mid \infty}T_v$. For $v \in \Val\left(K\right)$, let $T\left(\mathcal{O}_v\right)$ denote the maximal compact subgroup of $T\left(K_v\right)$.
\begin{mydef}
Let $v \in \Val\left(K\right)$ and $w \in \Val\left(E\right)$ with $w \mid v$.

For $v \nmid \infty$ with ramification degree $e_v$ in $E/K$, define the maps
$$
\deg_{T,v}: T\left(K_v\right) \rightarrow X_*\left(T_v\right), \quad t_v \mapsto [\chi_v \mapsto v\left( \chi_v\left(t_v\right)\right)]
$$
and $\deg_{T,E,v} = e_v \deg_{T,v}$.

For $v \mid \infty$, define the maps
$$
\deg_{T,v}: T\left(K_v\right) \rightarrow X_*\left(T_v\right)_{\mathbb{R}}, \quad t_v \mapsto [\chi_v \mapsto \log\left|\chi_v\left(t_v\right)\right|_v] 
$$
and $\deg_{T,E,v} = [E_w:K_v] \deg_{T,v}$.

Finally, define the maps
$$
\deg_T = \sum_{v \in \Val\left(K\right)} \left(\log q_v\right) \deg_{T,v}, \quad \deg_{T,E} = \sum_{v \in \Val\left(K\right)} \left(\log q_w\right) \deg_{T,E,v}.
$$
\end{mydef}
\begin{lemma} \label{lplem} \cite[\S2.2]{FZH}, \cite[\S4.2]{NVF} Let $v \in \Val\left(K\right)$, and let $f$ be either $\deg_{T,v}$ or $\deg_{T,E,v}$.
\begin{enumerate}[label=(\roman*)]
\item \label{lplem1} If $v$ is non-archimedean, then we have the exact sequence
$$
0 \rightarrow T\left(\mathcal{O}_v\right) \rightarrow T\left(K_v\right) \xrightarrow{f} X_*\left(T_v\right).
$$
The image of $f$ is open and of finite index. Further, if $v$ is unramified in $E$, then $f$ is surjective.
 \item \label{lplem2} If $v$ is archimedean, then we have the short exact sequence
$$
0 \rightarrow T\left(\mathcal{O}_v\right) \rightarrow T\left(K_v\right) \xrightarrow{f} X_*\left(T_v\right)_{\mathbb{R}} \rightarrow 0.
$$
Further, $f$ admits a canonical section.
\item \label{lplem3} Letting $g$ be either $\deg_T$ or $\deg_{T,E}$ and denoting its kernel by $T\left(\mathbb{A}_K\right)^1$, we have the split short exact sequence
$$
0 \rightarrow T\left(\mathbb{A}_K\right)^1 \rightarrow T\left(\mathbb{A}_K\right) \xrightarrow{g} X_*\left(T\right)_{\mathbb{R}} \rightarrow 0,
$$
hence we have an isomorphism
\begin{equation} \label{spliteqn}
T\left(\mathbb{A}_K\right) \cong T\left(\mathbb{A}_K\right)^1 \times X_*\left(T\right)_{\mathbb{R}}.
\end{equation}
\end{enumerate}
\end{lemma}
\begin{mydef}
Let $\chi$ be a character of $T\left(\mathbb{A}_K\right)$. We say that $\chi$ is \emph{automorphic} if it is trivial on $T\left(K\right)$. We say that $\chi$ is \emph{unramified at $v \in \Val \left(K\right)$} if $\chi_v$ is trivial on $T\left(\mathcal{O}_v\right)$, and we say that it is \emph{unramified} if it is unramified at every $v \in \Val \left(K\right)$.
\end{mydef}
The canonical sections of the maps $T\left(K_v\right) \xrightarrow{\deg_{T,v}} X_*\left(T_v\right)_{\mathbb{R}}$ for each $v \mid \infty$ from Lemma \hyperref[lplem2]{\ref*{lplem}\ref*{lplem2}} induce a canonical section of the composition
$$
T\left(\mathbb{A}_K\right) \rightarrow \prod_{v \mid \infty} T\left(K_v\right) \rightarrow X_*\left(T_\infty\right)_\mathbb{R}, 
$$
which in turn induces a ``type at infinity map''
\begin{equation} \label{taieq}
T\left(\mathbb{A}_K\right)^\wedge \rightarrow X^*\left(T_\infty\right)_{\mathbb{R}}, \quad \chi \mapsto \chi_{\infty}.
\end{equation}
Defining $\mathsf{K}_T = \prod_{v}T\left(\mathcal{O}_v\right)$, the splitting \eqref{spliteqn} for $g = \deg_T$ induces a map
$$
\left(T\left(\mathbb{A}_K\right)^1/T\left(K\right)\mathsf{K}_T\right) \rightarrow X^* \left(T_{\infty}\right)_{\mathbb{R}}
$$
which has finite kernel and image a lattice of codimension $\rank X^*\left(T\right)$ (see \cite[Lem.~4.52,~p.~96]{FZH}).
\begin{note}
When $T$ is anisotropic, we have $T\left(\mathbb{A}_K\right)^1 = T\left(\mathbb{A}_K\right)$ by Lemma \hyperref[lplem3]{\ref*{lplem}\ref*{lplem3}}, and then we see from the above that there is a map
$$
\left(T\left(\mathbb{A}_K\right)/T\left(K\right)\mathsf{K}_T\right) \rightarrow X^* \left(T_{\infty}\right)_{\mathbb{R}}
$$
with finite kernel and image a lattice of full rank.
\end{note}
%%
%%%%%%%%%%%%%%%%%%%%%%%%%%%%%%%%%%%%%%%%%%%%%%%%%%%%%%%%%%%%%%%%
\subsection{Hecke characters}
\begin{mydef}
A \emph{Hecke character} for $K$ is an automorphic character of $\mathbb{G}_{m,K}$.
\end{mydef}
Each Hecke character $\chi$ has a \emph{conductor} $q\left(\chi\right) \in \mathbb{N}$ (see \cite[\S5.10]{ANT}), which measures the ramification of $\chi$ at the non-archimedean places of $K$.
\begin{mydef}
A Hecke character is \emph{principal} if it is trivial on $\mathbb{G}_{m,K}\left(\mathbb{A}_K\right)^1$.
\end{mydef}
By Lemma \hyperref[lplem3]{\ref*{lplem}\ref*{lplem3}}, $\chi$ is principal if and only if $ \chi = \|\cdot\|^{i\theta}$ for some $\theta \in \mathbb{R}$, where $\|\cdot\|$ denotes the adelic norm map, i.e.\
$$
\|\cdot\|: \mathbb{A}_K^* \rightarrow S^1, \quad \left(x_v\right)_v \mapsto \prod_{v \in \Val\left(K\right)}|x_v|_v.
$$
\begin{mydef}
The \emph{(Hecke) $L$-function} $L\left(\chi,s\right)$ of a Hecke character $\chi$ is
$$
L\left(\chi,s\right) = \prod_v \left(1-\frac{\chi_v\left(\pi_v\right)}{q_v^s}\right)^{-1},
$$
where the product is taken over all places $v \nmid \infty$ at which $\chi$ is unramified.

The \emph{Dedekind zeta function} of $K$ is
$$
\zeta_K\left(s\right) = L\left(1,s\right).
$$
\end{mydef}
Given a Hecke character $\chi$ for a number field $L$ and $w \in \Val \left(L\right)$, we denote by $L_w\left(\chi,s\right)$ the local factor at $w$ for the Euler product of $L\left(\chi,s\right)$, i.e.\
$$
L_w\left(\chi,s\right) =
\begin{cases}
\left(1 - \frac{\chi_w \left(\pi_w\right)}{q_w^s}\right)^{-1} \textrm{ if $w \nmid \infty$ and $\chi$ is unramified at $w$}, \\
1 \textrm{ otherwise}.
\end{cases}
$$
When working over the field $L \supset K$, we define $L_v\left(\chi,s\right)$ for each $v \in \Val\left(K\right)$ by 
$$
L_v\left(\chi,s\right) = \prod_{w \mid v} L_w\left(\chi,s\right).
$$
\begin{theorem} \label{hlfthm} \cite[\S6]{ENAZ} 
The $L$-function of a Hecke character $\chi$ admits a meromorphic continuation to $\mathbb{C}$. If $\chi = \|\cdot\|^{i \theta}$ for some $\theta \in \mathbb{R}$, then this continuation admits a single pole of order $1$ at $s = 1+ i\theta$. Otherwise, it is holomorphic.
\end{theorem}
\begin{mydef} \label{vecnormdef}
Let $\psi$ be a character of $\prod_{v | \infty} K_v^*$. The restriction of $\psi$ to each $\mathbb{R}_{>0} \subset K_v^*$ is of the form $x \mapsto |x|^{i\kappa_v}$ for some $\kappa_v \in \mathbb{R}$. We define
$$
\|\psi\| = \max_{v \mid \infty} |\kappa_v|.
$$
\end{mydef}
\begin{lemma} \label{lflem} \cite[Exercise~3,~\S5.2,~p.~100]{ANT}
Let $\chi$ be a non-principal Hecke character of $K$, let $C$ be a compact subset of $\Re s \geq 1$ and let $\varepsilon > 0$. Then
$$
L\left(\chi,s\right) \ll_{\varepsilon, C} q\left(\chi\right)^{\varepsilon} \left(1+ \|\chi_{\infty}\|\right)^{\varepsilon}, \quad \left(s-1\right)\zeta_K\left(s\right) \ll_C 1, \quad s \in C.
$$
\end{lemma}
\begin{mydef}
Let $E/K$ be Galois, let $\chi$ be a Hecke character for $E$ and let $g \in \Gal \left(E/K\right)$. We define the \emph{(Galois) twist of $\chi$ by $g$} to be the character
$$
\chi^g: \mathbb{A}_E^* \rightarrow S^1, \quad
\left(t_w\right)_w \mapsto \chi\left(\left(g_w\left(t_w\right)\right)_{gw}\right).
$$
Here, $gw$ denotes the place of $E$ obtained by the action of $g$ on $\Val\left(E\right)$, and $g_w: E_w \rightarrow E_{gw}$ is the induced map on completions. One may easily verify that $\chi^g$ is trivial on $E^*$, hence it is also a Hecke character for $E$.
\end{mydef}
%%
%%%%%%%%%%%%%%%%%%%%%%%%%%%%%%%%%%%%%%%%%%%%%%%%%%%%%%%%%%%%%%%%%%%%%%%%%%%%%%%%%%%%%%%%%%%%%%%%%%%%%%%%%%%%%%%%%%%%%%%%%%%%%%%%
\section{The norm torus} \label{sectionTNT}
%%%%%%%%%%%%%%%%%%%%%%%%%%%%%%%%%%%%%%%%%%%%%%%%%%%%%%%%%%%%%%%%%%%%%%%%%%%%%%%%%%%%%%%%%%%%%%%%%%%%%%%%%%%%%%%%%%%%%%%%%%%%%%%%
In this section, we fix an extension of number fields $L/K$ of degree $d \geq 2$ with Galois closure $E$ and a $K$-basis $\om = \{\omega_0,\dots,\omega_{d-1}\}$. We write $N_{\om}\left(x_0,\dots,x_{d-1}\right)$ for the norm form corresponding to $\om$, and $G = \Gal\left(E/K\right)$. From the equality
\begin{equation} \label{normeq}
N_{\om}\left(x_0, \dots , x_{d-1}\right) = \prod_{g \in G/\Gal\left(E/L\right)}\left(x_0 g\left(\omega_0\right) + \dots + x_{d-1} g\left(\omega_{d-1}\right)\right),
\end{equation}
we see that $N_{\om}$ is irreducible over $K$ and has splitting field $E$.
We denote by $T$ the norm torus $T_{\om} = \mathbb{P}^{d-1}_K\setminus Z\left(N_{\om}\right)$. As noted in \cite[\S1.2]{NVF}, $\mathbb{P}^{d-1}_K$ is a toric variety with respect to $T$, and $T \cong R_{L/K}\mathbb{G}_m/\mathbb{G}_m$ is anisotropic. Since its boundary is $Z\left(N_{\om}\right)$, its splitting field is $E$. We have the short exact sequence
\begin{equation} \label{toruseq}
0 \rightarrow \mathbb{G}_m \rightarrow R_{L/K}\mathbb{G}_m \rightarrow T \rightarrow 0.
\end{equation}
\begin{note} \label{torhecnote}
The isomorphisms $T\left(\mathbb{A}_K\right) \cong \mathbb{A}_L^*/\mathbb{A}_K^*$ and $T\left(K\right) \cong L^*/K^*$ follow from Hilbert's Theorem 90 \cite[Prop.~IV.3.8,~p.~283]{ANTN} by applying Galois cohomology to \eqref{toruseq}. They allow us to interpret an automorphic character $\chi$ of $T$ as a Hecke character for $L$, and we will do so frequently. In fact, distinct automorphic characters of $T$ correspond to distinct Hecke characters of $L$ by \cite[Cor.~1.4.16,~p.~606]{RPBH} and \cite[Thm.~3.1.1,~p.~619]{RPBH}. Since $T\left(K_v\right) \cong \left(\prod_{w \mid v}L_w^*\right)/K_v^*$ for each $v \in \Val\left(K\right)$, we see that, if $\chi$ is unramified at $v$, then it is unramified as a Hecke character at all $w \mid v$. In particular, if $\chi$ is unramified at $v$ and $v$ is unramified in $L/K$, then $\prod_{w \mid v} \chi_w\left(\pi_w\right) = 1$, since $\pi_v$ is a uniformiser for $L_w$ for each $w \mid v$.
\end{note}
%%
%%%%%%%%%%%%%%%%%%%%%%%%%%%%%%%%%%%%%%%%%%%%%%%%%%%%%%%%%%%%%%%%
\subsection{Geometry} \label{subsectionG}
In this section we study fan-theoretic objects related to $T$. We begin by describing the fan $\Sigma \subset X_*\left(\overline{T}\right)_{\mathbb{R}}$ associated to the equivariant compactification $\mathbb{P}^{d-1}_K$ of $T$ and the associated piecewise-linear function $\varphi_{\Sigma}$ (see \cite[\S1.2]{RPBH}) used to define the Batyrev--Tschinkel height function. 

Denoting by $l_0\left(x\right),\dots,l_{d-1}\left(x\right) \in E[x]$ the $E$-linear factors of $N_{\om}\left(x\right)$, we have the $E$-isomorphism
$$
\begin{aligned}
\Phi: \overline{T} = \mathbb{P}^{d-1} \setminus \bigcup_{i=0}^{d-1}Z\left(l_i\right) \xrightarrow{\sim} \mathbb{G}_m^{d-1} = \mathbb{P}^{d-1} \setminus \bigcup_{j=0}^{d-1}Z\left(x_j\right), \\
[x_0,\dots,x_{d-1}] \mapsto [l_0\left(x\right),\dots,l_{d-1}\left(x\right)].
\end{aligned}
$$
By \cite[\S1.1]{TVSE}, the fan associated to $\mathbb{P}^{d-1}_E$ as a compactification of $\mathbb{G}_{m,E}^{d-1}$ is the fan whose $r$-dimensional cones are generated by the $r$-fold subsets of $\{e_0',\dots,e_{d-1}'\}$ for $0 \leq r \leq d-1$, where $e_i' \in X_*\left(\mathbb{G}_m^{d-1}\right) \cong \Hom\left(\mathbb{G}_m,\mathbb{G}_m^{d-1}\right)$ is defined by
$$
e'_i: \mathbb{G}_m \rightarrow \mathbb{G}_m^{d-1}, \quad t \mapsto [x_{0,i}\left(t\right),\dots,x_{d-1,i}\left(t\right)],
$$
where in turn
$$
x_{j,i}\left(t\right) = 
\begin{cases}
t \textrm{ if } i=j, \\
1 \textrm{ otherwise.}
\end{cases}
$$
\begin{mydef}
Set $e_i = \Phi^{-1} \circ e'_i$ for $i=0,\dots,d-1$, and define $\Sigma$ to be the fan whose $r$-dimensional cones are generated by the $r$-fold subsets of $\{e_0,\dots,e_{d-1}\}$ for $0 \leq r \leq d-1$.
\end{mydef}
It follows that $\Sigma$ is the fan associated to $\mathbb{P}^{d-1}_E$ as a compactification of $T_E$. Also, we see that $\sum_{i=0}^{d-1}e_i = 0$ and that $\{e_1,\dots,e_{d-1}\}$ is the dual of the basis $\{m_1,\dots,m_{d-1}\}$ of $X^*\left(\overline{T}\right)$, where $m_i\left(x\right) = \frac{l_i\left(x\right)}{l_0\left(x\right)}$ for $i=1,\dots,d-1$.

We now show that the action of $G$ on $\Sigma\left(1\right)$ is compatible with its action on the $E$-linear factors of $N_{\om}$. Denote by $*$ the action of $G$, and set $l_{g\left(i\right)} = g * l_i$.
\begin{lemma} \label{galactlem}
For all $g \in G$ and $i=0,\dots,d-1$, we have
$$
g * e_i = e_{g\left(i\right)}.
$$
\end{lemma}
\begin{proof}
Let $g \in G$. It suffices to show that
\begin{equation} \label{galacteqn}
\left(g * e_i\right)\left(m_j\right) = e_{g\left(i\right)}\left(m_j\right)
\end{equation}
for all $i \in \{0,\dots,d-1\}$ and $j \in \{1,\dots,d-1\}$. Note that, for any $i,j,k \in \{0,\dots,d-1\}$, we have 
\begin{equation} \label{krondelteqn}
e_i\left(\frac{l_j}{l_k}\right) = \delta_{ij} - \delta_{ik},
\end{equation}
where $\delta_{ij}$ is the Kronecker delta symbol, defined by
$$
\delta_{ij} =
\begin{cases}
1 \textrm{ if } i=j, \\
0 \textrm{ otherwise.}
\end{cases}
$$
Then \eqref{galacteqn} becomes
$$
\delta_{ig^{-1}\left(j\right)} - \delta_{ig^{-1}\left(0\right)} = \delta_{g\left(i\right)j} - \delta_{g\left(i\right)0},
$$
which clearly holds.
\end{proof}

By \cite[Thm.~1.22,~p.~217]{TVSE}, $\Sigma$ is the fan associated to the compactification $\mathbb{P}^{d-1}_K = \mathbb{P}^{d-1}_E / G$ of $T$ over $K$.

By \cite[Prop.~1.2.12,~p.~597]{RPBH}, the line bundle $L\left(\varphi_{\Sigma}\right)$ associated to the piecewise-linear function $\varphi_{\Sigma}: X_*\left(\overline{T
}\right)_{\mathbb{R}} \rightarrow \mathbb{R}$ (see \cite[Prop.~1.2.9,~p.~597]{RPBH}) defined by $\varphi_{\Sigma}\left(e_i\right) = 1$ for all $i=0,\dots,d-1$ is the anticanonical bundle $-K_{\mathbb{P}^{d-1}}$.

It follows from the above that $G$ acts transitively on $\Sigma\left(1\right) = \{\langle e_0 \rangle, \dots, \langle e_{d-1} \rangle\}$. For $v \in \Val \left(K\right)$ non-archimedean, let $G_v$ denote the associated decomposition subgroup of $G$. By the proof of \cite[Thm.~3.1.3,~p.~619]{RPBH}, the $G_v$-orbits of $\Sigma\left(1\right)$ are in bijection with the places of $L$ over $v$, and the length of the $G_v$-orbit corresponding to a place $w \mid v$ is its inertia degree.
\begin{proposition} \label{cocharprop}
Let $v \in \Val\left(K\right)$ be non-archimedean with ramification degree $e_v$ in $E/K$, and let
$$
\Sigma\left(1\right) = \bigcup_{w \mid v}\Sigma_w\left(1\right)
$$
denote the decomposition of $\Sigma\left(1\right) = \{\langle e_0 \rangle, \dots, \langle e_{d-1} \rangle\}$ into $G_v$-orbits. For each $w \mid v$, let $n_w$ be the sum of the elements of $\Sigma_w\left(1\right)$ and let $f_w\left(x\right)$ be the product of the linear factors in the $G_v$-orbit of $\{l_0,\dots,l_{d-1}\}$ corresponding to $\Sigma_w\left(1\right)$ by Lemma \ref{galactlem}.
Then the map $\deg_{T,E,v}: T\left(K_v\right) \rightarrow X_*\left(T_v\right)$ is given by
$$
t_v \mapsto e_v \sum_{w \mid v} \frac{v\left(f_w\left(t_v\right)\right)}{\deg f_w}n_w.
$$
\end{proposition}
\begin{proof}
The image of $t_v$ in $X_*\left(T_v\right) \cong X_*\left(\overline{T}\right)^{G_v}$ under $\deg_{T,E,v}$ is the cocharacter
$$
\varphi_{t_v}: X^*\left(T_v\right) \rightarrow \mathbb{Z}, \quad
\lambda \mapsto e_v v\left(\lambda\left(t_v\right)\right).
$$
We first show that $\{n_w : w \mid v\}$ spans $X_*\left(\overline{T}\right)^{G_v}$. Given $g \in G$ and $\sigma = \sum_{i=0}^{d-1} a_i e_i$, we have $g * \sigma = \sum_{i=0}^{d-1} a_{g^{-1}\left(i\right)}e_i$, so $g * \sigma = \sigma$ if and only if there exists $r_g \in \mathbb{Z}$ such that $a_i = a_{g^{-1}\left(i\right)} + r_g$ for all $i \in \{0,\dots,d-1\}$. Setting $s = \#G$, we have
$$
a_i = a_{g^s\left(i\right)} = a_{g^{s-1}\left(i\right)} + r_g = \dots = a_i + s r_g,
$$
hence $r_g = 0$. We deduce that $\sigma \in \Sigma^{G_v}$ if and only if $a_i = a_j$ for all $e_i, e_j$ in the same $G_v$-orbit, so the result follows. Moreover, we observe that $\sum_{w \mid v}a_w n_w = \sum_{w \mid v}b_w n_w$ if and only if there exists $r \in \mathbb{Z}$ such that $b_w = a_w + r$ for all $w \mid v$, since there is a unique expression for $\sigma \in X_*\left(\overline{T}\right)$ in the form $\sigma = \sum_{i=0}^{d-1} c_i e_i$ where $c_d = 0$.

Now, write
$$
\varphi_{t_v} = \sum_{w \mid v}{\alpha_w n_w}.
$$
Define $\mu_i \in X^*\left(\overline{T}\right)$ and $\lambda_w \in X^*\left(\overline{T}\right)^{G_v}$ for all $i \in \{0,\dots,{d-1}\}$ and all $w \mid v$ by
$$
\mu_i\left(x\right) = \frac{l_i\left(x\right)^d}{N_{\om}\left(x\right)}, \quad \lambda_w\left(x\right)  = \prod_{e_i \in \Sigma_w} \mu_i\left(x\right) = \frac{f_w\left(x\right)^d}{N_{\om}\left(x\right)^{\deg f_w}}.
$$
By \eqref{krondelteqn}, we have
$$
e_i\left(\mu_j\right) = 
\begin{cases}
d -1 \textrm { if } i = j, \\
-1 \textrm{ otherwise}.
\end{cases}
$$
Then, setting $d_w = \deg f_w$, we see that
$$
n_w\left(\lambda_{w'}\right) = 
\begin{cases}
d d_w - d_w^2  \textrm{ if } w = w', \\
-d_w d_{w'} \textrm{ otherwise},
\end{cases}
$$
so we deduce that
\begin{equation} \label{chareq1}
e_v v\left(\lambda_w\left(t_v\right)\right) = dd_w \alpha_w - d_w\sum_{w' \mid v} d_{w'} \alpha_{w'}
\end{equation}
for all $w \mid v$.
On the other hand, we have
\begin{equation} \label{chareq2}
e_v v\left(\lambda_w\left(t_v\right)\right) = e_v d v\left(f_w\left(t_v\right)\right) - e_v d_w\sum_{w' \mid v} v\left(f_{w'}\left(t_v\right)\right).
\end{equation}
Set $\beta_w = d_w \alpha_w - e_v v\left(f_w\left(t_v\right)\right)$. 
Combining \eqref{chareq1} and \eqref{chareq2}, we obtain
$$
d\beta_w = d_w \sum_{w' \mid v}\beta_{w'},
$$
hence $\beta_{w'} = \frac{d_{w'}}{d_w}\beta_w$ for all $w \mid v$, $w' \mid v$.
Since $K_v \cong E_w^{G_v}$ for any $w \mid v$, it follows that $d_w \mid v\left(f_w\left(t_v\right)\right)$, so $\beta_w \in d_w\mathbb{Z}$ for all $w \mid v$.
We deduce that there exists an integer $n \in \mathbb{Z}$ such that, for all $w \mid v$, we have $\beta_w = d_w n$, hence
$$
\alpha_w = e_v \frac{v \left(f_w\left(t_v\right)\right)}{\deg f_w} + n.
$$
Since $\sum_{w \mid v} n_w = \sum_{i}e_i = 0$, we conclude that
\[
\varphi_{t_v} = e_v \sum_{w \mid v} \frac{v\left(f_w\left(t_v\right)\right)}{\deg f_w}n_w. \qedhere
\]
\end{proof}
We now study polynomials introduced by Batyrev and Tschinkel in \cite[\S2.2]{RPBH}, which play a key role in the analysis of local Fourier transforms in Section \ref{sectionWCP}.
\begin{mydef}
Let $v \in \Val \left(K\right)$ be non-archimedean, and let $\Sigma\left(1\right) = \bigcup_{i=1}^l\Sigma_i\left(1\right)$ be the decomposition of $\Sigma\left(1\right)$ into $G_v$-orbits. Let $d_i$ be the cardinality of $\Sigma_i\left(1\right)$. For each $\Sigma_i\left(1\right)$, define an independent variable $u_i$. Let $\sigma \in \Sigma^{G_v}$, and let $\Sigma_{i_1}\left(1\right) \cup \dots \cup \Sigma_{i_k}\left(1\right)$ be the set of $1$-dimensional faces of $\sigma$. We define the rational function
$$
R_{\sigma,v}\left(u_1,\dots,u_l\right) = \frac{u_{i_1}^{d_{i_1}}\dotsm u_{i_k}^{d_{i_k}}}{\left(1-u_{i_1}^{d_{i_1}}\right)\dotsm\left(1-u_{i_k}^{d_{i_k}}\right)},
$$
and we define the polynomial $Q_{\Sigma,v}\left(u_1,\dots,u_l\right)$ by
$$
\frac{Q_{\Sigma,v}\left(u_1,\dots,u_l\right)}{\left(1-u_1^{d_1}\right)\dotsm\left(1-u_l^{d_l}\right)} = \sum_{\sigma \in \Sigma^{G_v}}R_{\sigma,v}\left(u_1,\dots,u_l\right).
$$
\end{mydef}
\begin{proposition} \label{fanprop}
For all non-archimedean valuations $v \in \Val \left(K\right)$, we have
$$
Q_{\Sigma,v}\left(u_1,\dots,u_l\right) = 1 - u_1^{d_1}\dotsm u_l^{d_l}.
$$
\end{proposition}
\begin{proof}
Observe that the $G_v$-invariant cones in $\Sigma$ are precisely those cones generated by a set of $1$-dimensional cones of the form $\Sigma_{i_1}\left(1\right) \cup \dots \cup \Sigma_{i_k}\left(1\right)$ for some $i_1,\dots,i_k \in \{1,\dots,l\}$ pairwise distinct with $k < l$. From this observation, we deduce that
$$
\frac{Q_{\Sigma,v}\left(u_1,\dots,u_l\right)}{\left(1-u_1^{d_1}\right)\dotsm\left(1-u_l^{d_l}\right)} = \sum_{k=1}^{l-1} \sum_{\substack{i_1,\dots,i_k \in \{1,\dots,l\} \\ \textrm{ pairwise distinct}}}\frac{u_{i_1}^{d_{i_1}}\dotsm u_{i_k}^{d_{i_k}}}{\left(1-u_{i_1}^{d_{i_1}}\right)\dotsm\left(1-u_{i_k}^{d_{i_k}}\right)}.
$$
In particular, we see that
$$
Q_{\Sigma,v}\left(u_1,\dots,u_l\right) = \sum_{\left(t_1,\dots,t_l\right) \in \{0,1\}^l} \prod_{i=1}^l\left(t_i + \left(1-2t_i\right)u_i^{d_i}\right) - u_1^{d_1} \dotsm u_l^{d_l},
$$
so it suffices to prove that
\begin{equation} \label{bineq}
\sum_{\left(t_1,\dots,t_l\right) \in \{0,1\}^l} \prod_{i=1}^l\left(t_i + \left(1-2t_i\right)u_i^{d_i}\right) = 1.
\end{equation}
Splitting the sum into two smaller sums for $t_1 = 0$ and $t_1 = 1$, we obtain
$$
\begin{aligned}
& \sum_{\left(t_1,\dots,t_l\right) \in \{0,1\}^l} \prod_{i=1}^l\left(t_i + \left(1-2t_i\right)u_i^{d_i}\right) \\
= & \left(u_1^{d_1} + \left(1-u_1^{d_1}\right)\right)\sum_{\left(t_2,\dots,t_n\right) \in \{0,1\}^{l-1}} \prod_{i=2}^l \left(t_i + \left(1-2t_i\right)u_i^{d_i}\right) \\
= & \sum_{\left(t_2,\dots,t_n\right) \in \{0,1\}^{l-1}} \prod_{i=2}^l \left(t_i + \left(1-2t_i\right)u_i^{d_i}\right).
\end{aligned}
$$
Repeating this process for each variable $t_2,\dots,t_l$, we deduce \eqref{bineq}.
\end{proof}
%%
%%%%%%%%%%%%%%%%%%%%%%%%%%%%%%%%%%%%%%%%%%%%%%%%%%%%%%%%%%%%%%%%
\subsection{Haar measures and volume} \label{section:HMV}
Let $\omega$ be an invariant $d$-form on $T$. By a classical construction (see \cite[\S2.1.7]{IIVA}), $\omega$ gives rise to a Haar measure $|\omega|_v$ on $T\left(K_v\right)$ for each $v \in \Val\left(K\right)$. In \cite[\S3.3]{AAT}, Ono constructs the convergence factors
$$
c_v = 
\begin{cases}
L_v\left(X^*\left(\overline{T}\right),1\right)^{-1} \textrm{ if } v \nmid \infty, \\
1 \textrm{ if } v \mid \infty.
\end{cases}
$$
Here, $L_v\left(X^*\left(\overline{T}\right),s\right)$ is the local factor at $v$ of the Artin $L$-function $L\left(X^*\left(\overline{T}\right),s\right)$.
Defining $\mu_v = c_v^{-1}|\omega|_v$, the product of the $\mu_v$ converges to give a Haar measure $\mu$ on $T\left(\mathbb{A}_K\right)$, which is independent of $\omega$ by the product formula.
\begin{note}
From the short exact sequence \eqref{toruseq}, we obtain
$$
L\left(X^*\left(\overline{T}\right),s\right) = \frac{\zeta_L \left(s\right)}{\zeta_K\left(s\right)}.
$$
\end{note}
\begin{lemma} \label{vollem}
With respect to the Haar measure $\mu$, we have
$$
\vol \left(T\left(\mathbb{A}_K\right)/T\left(K\right)\right) = d \frac{\Res_{s=1}\zeta_L\left(s\right)}{\Res_{s=1}\zeta_K\left(s\right)}.
$$
\end{lemma}
\begin{proof}
By \cite[\S3.5]{AAT} and \cite[Main~Thm.,~p.~68]{OTNAT}, we have
$$
\vol \left(T\left(\mathbb{A}_K\right)^1/T\left(K\right)\right) = \frac{|\Pic T|}{|\Sh \left(T\right)|} L\left(X^*\left(\overline{T}\right),1\right),
$$
where $\Sh\left(T\right)$ is the Tate-Shafarevich group of $T$, i.e.\
$$
\Sh \left(T\right) = \ker\left(H_{\textrm{\'et}}^1\left(K,T\right) \rightarrow \prod_{v \in \Val \left(K\right)} H_{\textrm{\'et}}^1\left(K_v,T\right)\right).
$$
By \cite[Prop.~8.3,~p.~58]{GB} and \cite[Cor.~4.6,~p.~2568]{NVF}, the rationality of $T$ implies that $\Sh \left(T\right)$ is trivial. Further, we have $\Pic T \cong \mathbb{Z}/d\mathbb{Z}$ (see \cite[Prop.~II.6.5(c),~p.~133]{AG}). Since $\zeta_K\left(s\right)$ and $\zeta_L\left(s\right)$ both have a simple pole at $s=1$, we have $L\left(X^*\left(\overline{T}\right),1\right) = \frac{\Res_{s=1}\zeta_L\left(s\right)}{\Res_{s=1}\zeta_K\left(s\right)
}$. Finally, as $T$ is anisotropic, we have $T\left(\mathbb{A}_K\right)^1 = T\left(\mathbb{A}_K\right)$.
\end{proof}
%%
%%%%%%%%%%%%%%%%%%%%%%%%%%%%%%%%%%%%%%%%%%%%%%%%%%%%%%%%%%%%%%%%
%%%%%%%%%%%%%%%%%%%%%%%%%%%%%%%%%%%%%%%%%%%%%%%%%%%%%%%%%%%%%%%%
\section{Heights and indicator functions} \label{sectionHIF}
%%%%%%%%%%%%%%%%%%%%%%%%%%%%%%%%%%%%%%%%%%%%%%%%%%%%%%%%%%%%%%%%
In this section we define functions which allow us to use harmonic analysis to study weak Campana points. Let $L/K$ be an extension of number fields with $K$-basis $\om = \{\omega_0,\dots,\omega_{d-1}\}$ and Galois closure $E/K$. When $L=E$, for any $i,j \in \{0,\dots,d-1\}$ and $g \in G = \Gal\left(E/K\right)$, write
$$
\omega_i \cdot \omega_j = \sum_{k=0}^{d-1}a^{ij}_{k}\omega_k, \quad
g\left(\omega_i\right) = \sum_{k=0}^{d-1}b^g_{k}\omega_k, \quad
1 = \sum_{k=0}^{d-1}c_k \omega_k.
$$
\begin{mydef} \label{somdef}
When $L=E$, we define $S\left(\om\right)$ to be the minimal subset of $\Val\left(K\right)$ containing $S_{\infty}$ such that $a^{ij}_k, b^g_k,c_k \in \mathcal{O}_v$ for all $v \not\in S\left(\om\right)$, $i,j,k \in \{0,\dots,d-1\}$ and $g \in G$. Otherwise, we define $S\left(\om\right)$ to be the minimal subset of $\Val\left(K\right)$ containing $S_{\infty}$ such that $N_{\om}$ is an irreducible polynomial over $\mathcal{O}_{K,S\left(\om\right)}$.
\end{mydef}
\begin{remark} \label{somrem}
When $L=E$, by \eqref{normeq}, the $S\left(\om\right)$-integrality of the $a^{ij}_k$ and $b^g_k$ implies that $N_{\om}$ is defined over $\mathcal{O}_{K,S\left(\om\right)}$, while the $S\left(\om\right)$-integrality of the $c_k$ implies that the coefficients of $N_{\om}$ are not all divisible by some $\alpha \in \mathcal{O}_{K,S\left(\om\right)} \setminus \mathcal{O}_{K,S\left(\om\right)}^*$. Since $N_{\om}$ is irreducible over $K$, we deduce that $N_{\om}$ is irreducible over $\mathcal{O}_{K,S\left(\om\right)}$, hence, for any $L$, the Zariski closure of $Z\left(N_{\om}\right)$ in $\mathbb{P}^{d-1}_{\mathcal{O}_{K,S\left(\om\right)}}$ is $\Proj \mathcal{O}_{K,S\left(\om\right)}[x_0,\dots,x_{d-1}]/\left(N_{\om}\right)$.
\end{remark}
From now on, we fix the model $\left(\mathbb{P}^{d-1}_{\mathcal{O}_{K,S\left(\om\right)}},\mathcal{D}^{\om}_m\right)$ for $\left(\mathbb{P}^{d-1}_K,\Delta_m^{\om}\right)$, where $\mathcal{D}^{\om}_m = \left(1-\frac{1}{m}\right)\Proj \mathcal{O}_{K,S\left(\om\right)}[x_0,\dots,x_{d-1}]/\left(N_{\om}\right)$. We denote by $\mathcal{D}^{\om}_{\textrm{red}}$ the support of $\mathcal{D}^{\om}_m$.
\begin{note}
In both the Galois and non-Galois cases, the conditions on $S\left(\om\right)$ ensure that we may take the ``obvious'' model above. The potentially stronger conditions in the Galois case (in which we obtain our results) ensure compatibility between intersection multiplicity and toric multiplication, as we shall see in Section \ref{section:IS}.
\end{note}
\begin{remark} \label{ribrem}
When $L=E$ and $\om$ is a relative integral basis, we get $S\left(\om\right) = S_{\infty}$, since every algebraic integer is expressible as an $\mathcal{O}_K$-linear combination of elements of a relative integral basis, and $\mathcal{O}_L$ is closed under multiplication and conjugation.
\end{remark}
\subsection{Definitions}
\begin{mydef} \label{bthtdef} \cite[\S2.1]{RPBH}
For each place $v$ of $K$, we define the \emph{local height function}
$$
H_v : T\left(K_v\right) \rightarrow \mathbb{R}_{>0}, \quad
t_v \mapsto e^{\varphi_{\Sigma}\left(\deg_{T,E,v}\left(t_v\right)\right)\log q_v}.
$$
We then define the \emph{global height function}
$$
H: T\left(\mathbb{A}_K\right) \rightarrow \mathbb{R}_{>0}, \quad
\left(t_v\right)_v \mapsto \prod_{v \in \Val\left(K\right)}H_v\left(t_v\right).
$$
\end{mydef}
\begin{mydef}
For each place $v \not\in S\left(\om\right)$, define the function
$$
H'_v: T\left(K_v\right) \rightarrow \mathbb{R}_{>0}, \quad
x \mapsto \frac{\max\{|x_i|_v^d\}}{|N_{\om}\left(x\right)|_v}.
$$
\end{mydef}
\begin{remark} \label{lbrem}
Note that $H'_v\left(x\right) \geq 1$ for all $x \in T\left(K_v\right)$. Indeed, one may always select $v$-adic coordinates $x_i$ such that $\max\{|x_i|_v\} = 1$, and $N_{\om}$ has coefficients in $\mathcal{O}_v$ by Remark \ref{somrem}, so, by the strong triangle inequality, we have $|N_{\om}\left(x\right)|_v \leq 1$.
\end{remark}
\begin{lemma} \label{sdeflem}
For all but finitely many places $v \not\in S\left(\om\right)$, we have $H'_v = H_v$.
\end{lemma}
\begin{proof}
Note that $H_v'$ is the local Weil function associated to the basis of global sections of $-K_{\mathbb{P}^{d-1}}$ consisting of all monomials of degree $d$ in \cite[Def.~2.1.1,~p.~606]{RPBH}. It is well-known (see \cite[\S2.2.3]{IIVA}) that two height functions corresponding to adelic metrisations of the same line bundle are equal over all but finitely many places.
\end{proof}
\begin{mydef} \label{sdef}
We define the finite set 
$$
S'\left(\om\right) = S\left(\om\right) \cup \{v \not\in S\left(\om\right) : H'_v \neq H_v\} \cup\{v \in \Val\left(K\right): E/K \textrm{ is ramified at $v$}\}.
$$
\end{mydef}
\begin{mydef} \label{inddef}
For each place $v \not\in S\left(\om\right)$, define the \emph{local indicator function}
$$
\phi_{m,v} : T\left(K_v\right) \rightarrow \{0,1\}, \quad
t_v \mapsto
\begin{cases}
1 \textrm{ if } H'_v\left(t_v\right) = 1 \textrm{ or } H'_v\left(t_v\right) \geq q_v^m, \\
0 \textrm { otherwise}. 
\end{cases}
$$
Setting $\phi_{m,v} = 1$ for $v \in S\left(\om\right)$, we then define the \emph{global indicator function}
$$
\phi_m : T\left(\mathbb{A}_K\right) \rightarrow \{0,1\}, \quad
\left(t_v\right)_v \mapsto \prod_{v \in \Val \left(K\right)} \phi_{m,v}\left(t_v\right).
$$
\end{mydef}
\begin{remark}\label{ctsrem}
Let $v \not\in S\left(\om\right)$ be a non-archimedean place of $K$. Since $H'_v$ is continuous with discrete image in $\mathbb{R}_{>0}$, its level sets are clopen. It follows that $\phi_{m,v}$ is continuous for all $v \in \Val \left(K\right)$. Also, since $\phi_{m,v}\left(T\left(\mathcal{O}_v\right)\right) = 1$ for all $v \not\in S'\left(\om\right)$ by Lemma \hyperref[lplem1]{\ref*{lplem}\ref*{lplem1}}, we see that $\phi_m$ is well-defined and continuous on $T\left(\mathbb{A}_K\right)$.
\end{remark}
\begin{lemma} \label{identlem}
The weak Campana $\mathcal{O}_{K,S\left(\om\right)}$-points of $\left(\mathbb{P}^{d-1}_K,\Delta^{\om}_m\right)$ are precisely the rational points $t \in T\left(K\right)$ such that $\phi_m\left(t\right) = 1$.
\end{lemma}
\begin{proof}
Take $v \not\in S\left(\om\right)$, and let $t_0,\dots,t_{d-1}$ be a set of $\mathcal{O}_v$-coordinates for $t \in T\left(K\right)$ with at least one $t_i \in \mathcal{O}_v^*$. Then we have
\[
H'_v\left(t\right) = \frac{1}{|N_{\om}\left(t_0,\dots,t_{d-1}\right)|_v} = q_v^{v\left(N_{\om}\left(t_0,\dots,t_{d-1}\right)\right)} = q_v^{n_v\left(\mathcal{D}^{\om}_{\textrm{red}},t\right)}. \qedhere
\]
\end{proof}
%%
%%%%%%%%%%%%%%%%%%%%%%%%%%%%%%%%%%%%%%%%%%%%%%%%%%%%%%%%%%%%%%%%
\subsection{Invariant subgroups} \label{section:IS}
%%%%%%%%%%%%%%%%%%%%%%%%%%%%%%%%%%%%%%%%%%%%%%%%%%%%%%%%%%%%%%%%
%%
For this section, let $L = E$ be Galois over $K$.
\begin{lemma} \label{smlem}
For all $v \not\in S\left(\om\right)$ and $x,y \in T\left(K_v\right)$, we have
$$
H'_v \left(x \cdot y\right) \leq H'_v \left(x\right) H'_v \left(y\right).
$$
\end{lemma}
\begin{proof}
Choose sets of projective coordinates $\{x_0,\dots,x_{d-1}\}$ and $\{y_0,\dots,y_{d-1}\}$ for $x$ and $y$ respectively. Note that
$$
\left(x_0 \omega_0 + \dots + x_{d-1} \omega_{d-1}\right) \cdot \left(y_0 \omega_0 + \dots + y_{d-1} \omega_{d-1}\right) = \left(z_0 \omega_0 + \dots + z_{d-1} \omega_{d-1}\right),
$$
where, for $a^{ij}_k \in \mathcal{O}_v$ as in Definition \ref{somdef}, we have 
$$
z_k = \sum_{i=0}^{d-1}\sum_{j=0}^{d-1} a^{ij}_{k} x_i y_j.
$$
Using $N_{\om}\left(x \cdot y\right) = N_{\om}\left(x\right)N_{\om}\left(y\right)$ and the strong triangle inequality, we deduce that
\begin{align*}
H'_v\left(x \cdot y\right) & = \frac{\max\{|\sum_{i=0}^{d-1}\sum_{j=0}^{d-1}a^{ij}_k x_i y_j|_v^{d}\}}{|N_{\om}\left(x \cdot y\right)|_v} \\
& \leq \max\{|a^{ij}_{k}|_v^d\} \frac{\max\{|x_i|_v^d\}}{|N_{\om}\left(x\right)|_v} \frac{\max\{|y_j|_v^{d}\}}{|N_{\om}\left(y\right)|_v} \leq H'_v\left(x\right) H'_v\left(y\right). \qedhere
\end{align*}
\end{proof}
\begin{lemma} \label{mplem}
For any place $v \not\in S\left(\om\right)$, the level set
$$
\mathcal{K}_v = \{t_v \in T\left(K_v\right) : H'_v\left(t_v\right) = 1\}
$$
is a subgroup of $T\left(\mathcal{O}_v\right)$.
\end{lemma}
\begin{proof}
From Proposition \ref{cocharprop} and Lemma \hyperref[lplem1]{\ref*{lplem}\ref*{lplem1}}, it is clear that $H_v'\left(t_v\right) = 1$ implies $t_v \in T\left(\mathcal{O}_v\right)$, so $\mathcal{K}_v \subset T\left(\mathcal{O}_v\right)$. It is also clear that $H'_v\left(1\right) = 1$, and closure under multiplication follows from Lemma \ref{smlem} and Remark \ref{lbrem}. It only remains to verify that $x \in \mathcal{K}_v$ implies $x^{-1} \in \mathcal{K}_v$. Let $x \in \mathcal{K}_v$, and choose coordinates $x_0,\dots,x_{d-1}$ with $\max\{|x_i|_v^d\}$ = 1. Since $H'_v\left(x\right) = 1$, we must have $|N_{\om}\left(x\right)|_v = 1$. Note that
$$
\left(x_0 \omega_0 + \dots + x_{d-1} \omega_{d-1}\right)^{-1} = \frac{1}{N_{\om}\left(x_0,\dots,x_{d-1}\right)}\prod_{\substack{g \in G \\ g \neq 1_G}}\left(x_0 g\left(\omega_0\right) + \dots + x_0 g\left(\omega_0\right)\right).
$$
Recursively applying Lemma \ref{smlem}, we obtain
$$
H_v'\left(x^{-1}\right) \leq \prod_{\substack{g \in G \\ g \neq 1_G}} H_v'\left(g\left(x\right)\right).
$$
By Remark \ref{lbrem}, it suffices to show that, for any $g \in G$, we have $H_v'\left(g\left(x\right)\right) = 1$. Since $N_{\om}\left(g\left(x\right)\right) = N_{\om}\left(x\right)$, it suffices by Remark \ref{lbrem} to show that $\max\{|g\left(x\right)_i|_v\} \leq 1$. This follows from the fact that $b^g_k \in \mathcal{O}_v$ since $v \not\in S\left(\om\right)$, see Definition \ref{somdef}.
\end{proof}
\begin{corollary} \label{invcor}
For every place $v \not\in S\left(\om\right)$, the function $H'_v$ is $\mathcal{K}_v$-invariant.
\end{corollary}
\begin{proof}
Take $x \in \mathcal{K}_v$, and let $y \in T\left(K_v\right)$. Then by Lemma \ref{smlem}, we have
$$
H'_v\left(x \cdot y\right) \leq H'_v\left(x\right) H'_v \left(y\right) = H'_v\left(y\right),
$$
while on the other hand, since $x^{-1} \in \mathcal{K}_v$ by Lemma \ref{mplem}, we have
$$
H'_v\left(y\right) = H'_v \left(x^{-1} \cdot \left(x \cdot y\right)\right) \leq H'_v \left(x^{-1}\right) H'_v \left(x \cdot y\right) = H'_v \left(x \cdot y\right),
$$
so we conclude that $H'_v \left(x \cdot y\right) = H'_v\left(y\right)$.
\end{proof}
\begin{lemma} \label{mcslem}
For each place $v \not\in S\left(\om\right)$, the functions $H_v$ and $\phi_{m,v}$ are both $\mathcal{K}_v$-invariant and $1$ on $\mathcal{K}_v$. Further, $\mathcal{K}_v$ is compact, open and of finite index in $T\left(\mathcal{O}_v\right)$. Moreover, when $v \not \in S'\left(\om\right)$, we have $\mathcal{K}_v = T\left(\mathcal{O}_v\right)$.
\end{lemma}
\begin{proof}
Let $v \not\in S\left(\om\right)$. By \cite[Thm.~2.1.6(i),~p.~608]{RPBH}, $H_v$ is $T\left(\mathcal{O}_v\right)$-invariant, hence trivial and invariant on all of $T\left(\mathcal{O}_v\right)$. By Corollary \ref{invcor}, the function $\phi_{m,v}$ is $\mathcal{K}_v$-invariant; since $\phi_{m,v}\left(1\right) = 1$, it is also trivial on $\mathcal{K}_v$.

Now, since $\mathcal{K}_v = \left(H_v'|_{T\left(\mathcal{O}_v\right)}\right)^{-1}\left(\{1\}\right)$, it is open. Since the cosets of an open subgroup form an open cover of a topological group, any open subgroup of a compact topological group is closed and of finite index. Then $\mathcal{K}_v \subset T\left(\mathcal{O}_v\right)$ is closed, hence compact, and of finite index.
Finally, we note that, when $v \not\in S'\left(\om\right)$, we have $H'_v = H_v$, and $H_v^{-1}\left(\{1\}\right) = T\left(\mathcal{O}_v\right)$ by Lemma \hyperref[lplem1]{\ref*{lplem}\ref*{lplem1}} and the equality $\varphi^{-1}_{\Sigma}\left(\{0\}\right) = \{0\}$, so $\mathcal{K}_v = T\left(\mathcal{O}_v\right)$.
\end{proof}
\begin{mydef} \label{kudef}
For each $v \in S\left(\om\right)$, set $\mathcal{K}_v = T\left(\mathcal{O}_v\right)$. Let $\mathcal{K} = \prod_{v \in \Val\left(K\right)} \mathcal{K}_v$, and let $\mathcal{U}$ be the group of automorphic characters of $T$ which are trivial on $\mathcal{K}$. 
\end{mydef}
%%
%%%%%%%%%%%%%%%%%%%%%%%%%%%%%%%%%%%%%%%%%%%%%%%%%%%%%%%%%%%%%%%%
\subsection{Height zeta function and Fourier transforms}
\begin{mydef}
For $\Re s \gg 0$, we define the \emph{height zeta function}
$$
Z_m :  \mathbb{C} \rightarrow \mathbb{C}, \quad
s \mapsto \sum_{x \in \mathbb{P}^{d-1}\left(K\right)} \frac{\phi_{m}\left(x\right)}{H\left(x\right)^s}.
$$
\end{mydef}
\begin{mydef}
Let $\mu_v$ and $\mu$ be the Haar measures introduced in Section \ref{section:HMV}. Let $f: T\left(\mathbb{A}_K\right) \rightarrow \mathbb{C}$ be a continuous function given as a product of local factors $f_v: T\left(K_v\right) \rightarrow \mathbb{C}$ such that $f_v\left(T\left(\mathcal{O}_v\right)\right) = 1$ for all but finitely many places $v \in \Val\left(K\right)$. For each place $v \in \Val\left(K\right)$ and each character $\chi$ of $T\left(\mathbb{A}_K\right)$, we define the \emph{local Fourier transform of $\chi_v$ with respect to $f_v$} to be
$$
\widehat{H}_v\left(f_v,\chi_v;-s\right) = \int_{T\left(K_v\right)}\frac{f_v\left(t_v\right)\chi_v\left(t_v\right)}{H_v\left(t_v\right)^s}d\mu_v
$$
for all $s \in \mathbb{C}$ for which the integral exists. We then define the \emph{global Fourier transform of $\chi$ with respect to $f$} to be
$$
\widehat{H}\left(f,\chi;-s\right) = \prod_{v \in \Val\left(K\right)} \widehat{H}_v\left(f_v,\chi_v;-s\right) = \int_{T\left(\mathbb{A}_K\right)}\frac{f\left(t\right)\chi\left(t\right)}{H\left(t\right)^s}d\mu.
$$
\end{mydef}
%%
%%%%%%%%%%%%%%%%%%%%%%%%%%%%%%%%%%%%%%%%%%%%%%%%%%%%%%%%%%%%%%%%%%%%%%%%%%%%%%%%%%%%%%%%%%%%%%%%%%%%%%%%%%%%%%%%%%%%%%%%%%%%%%%%
\section{Weak Campana points} \label{sectionWCP}
%%%%%%%%%%%%%%%%%%%%%%%%%%%%%%%%%%%%%%%%%%%%%%%%%%%%%%%%%%%%%%%%%%%%%%%%%%%%%%%%%%%%%%%%%%%%%%%%%%%%%%%%%%%%%%%%%%%%%%%%%%%%%%%%
In this section we prove Theorem \ref{mainthm}. Fix an extension of number fields $L/K$ of degree $d$ with $K$-basis $\om$, set $T = T_{\om}$ as in Section \ref{sectionTNT} and let $m \in \mathbb{Z}_{\geq 2}$.
%%%%%%%%%%%%%%%%%%%%%%%%%%%%%%%%%%%%%%%%%%%%%%%%%%%%%%%%%%%%%%%%
\subsection{Strategy}
%%%%%%%%%%%%%%%%%%%%%%%%%%%%%%%%%%%%%%%%%%%%%%%%%%%%%%%%%%%%%%%%
%%
Following \cite{NVF} and \cite{RPBH}, we will apply a Tauberian theorem \cite[Thm.~3.3.2,~p.~624]{RPBH} to our height zeta function $Z_m\left(s\right)$ in order to find an asymptotic for the number of weak Campana points of bounded height. By \emph{loc.\ cit.,} it suffices to show that $Z_m\left(s\right)$ is absolutely convergent for $\Re s > \frac{1}{m}$ and that $Z_m\left(s\right)\left(s-\frac{1}{m}\right)^{b\left(d,m\right)}$ admits an extension to a holomorphic function on $\Re s \geq \frac{1}{m}$ which is not zero at $s = \frac{1}{m}$.
In order to do this, we will apply the version of the Poisson summation formula given by Bourqui \cite[Thm.~3.35,~p.~64]{FZH}.
Formally applying this version with $\mathcal{G} = T\left(\mathbb{A}_K\right)$, $\mathcal{H} = T\left(K\right)$, $dg = d\mu$, $dh$ the discrete measure on $T\left(K\right)$ and $F\left(t\right) = \frac{\phi_m\left(t\right)}{H\left(t\right)^s}$ for some $s \in \mathbb{C}$ with $\Re s > \frac{1}{m}$ gives
\begin{equation} \label{formappeq}
 Z_m\left(s\right) = \frac{1}{\vol\left(T\left(\mathbb{A}_K\right)/T\left(K\right)\right)} \sum_{\chi \in \left(T\left(\mathbb{A}_K\right)/T\left(K\right)\right)^{\wedge} } \widehat{H}\left(\phi_m, \chi; -s\right).
\end{equation}
%%%%%%%%%%%%%%%%%%%%%%%%%%%%%%%%%%%%%%%%%%%%%%%%%%%%%%%%%%%%%%%%
\subsection{Analytic properties of Fourier transforms}
\begin{lemma} \label{aclem}
For any place $v \in \Val\left(K\right)$, any character $\chi_v$ of $T\left(K_v\right)$ and any $\epsilon > 0$, the local Fourier transform $\widehat{H}_v\left(\phi_{m,v},\chi_v;-s\right)$ is absolutely convergent and is bounded uniformly (in terms of $\epsilon$ and $v$) on $\Re s \geq \epsilon$.
\end{lemma}
\begin{proof}
Let $\Re s \geq \epsilon$. Since
$$
|\widehat{H}_v\left(\phi_{m,v},\chi_v;-s\right)| \leq \int_{T\left(K_v\right)}\left|\frac{\phi_{m,v}\left(t_v\right)\chi_v\left(t_v\right)}{H_v\left(t_v\right)^s}\right|d\mu_v \leq \widehat{H}_v\left(1,1;-\epsilon\right),
$$
it suffices to prove that $\widehat{H}_v\left(1,1;-\epsilon\right)$ is convergent. For $v \mid \infty$, this follows from \cite[Prop.~2.3.2,~p.~614]{RPBH}, so assume that $v \nmid \infty$. The following argument is essentially the one in \cite[Rem.~2.2.8,~p.~613]{RPBH}, but we fill in the details for the sake of clarity.

Since $H_v$ and $d\mu_v$ are $T\left(\mathcal{O}_v\right)$-invariant and $\int_{T\left(\mathcal{O}_v\right)}d\mu_v = 1$, we have
\begin{equation*}
\begin{aligned}
\widehat{H}_v\left(1,1;-\epsilon\right) = \int_{T\left(K_v\right)}\frac{1}{H_v\left(t_v\right)^\epsilon}d\mu_v = \sum_{\overline{t_v} \in T\left(K_v\right)/T\left(\mathcal{O}_v\right)}\frac{1}{H_v\left(\overline{t_v}\right)^{\epsilon}}.
\end{aligned}
\end{equation*}
Now, by Lemma \hyperref[lplem1]{\ref*{lplem}\ref*{lplem1}}, $T\left(K_v\right)/T\left(\mathcal{O}_v\right)$ can be identified with a sublattice of finite index in $X_*\left(T_v\right)$, and this sublattice coincides with $X_*\left(T_v\right)$ when $v$ is unramified in $L/K$. Then we see that, interpreting $H_v$ as a function on $X_*\left(T_v\right)$, we have
$$
\sum_{\overline{t_v} \in T\left(K_v\right)/T\left(\mathcal{O}_v\right)}\frac{1}{H_v\left(\overline{t_v}\right)^{\epsilon}} \leq \sum_{n_v \in X_*\left(T_v\right)}\frac{1}{H_v\left(n_v\right)^{\epsilon}},
$$
and the proof of \cite[Thm.~2.2.6,~p.~611]{RPBH} and Proposition \ref{fanprop} give
$$
\sum_{n_v \in X_*\left(T_v\right)}\frac{1}{H_v\left(n_v\right)^{\epsilon}} = \left(1-\frac{1}{q_v^{d\epsilon}}\right)\prod_{w \mid v}\left(1-\frac{1}{q_w^{\epsilon}}\right)^{-1},
$$
so we deduce that $\widehat{H}_v\left(1,1;-\epsilon\right)$ is convergent, and this concludes the proof.
\end{proof}
\begin{lemma} \label{nontrivlem}
For any $v \in \Val\left(K\right)$, the local Fourier transform $\widehat{H}_v\left(\phi_{m,v},1;-s\right)$ is non-zero for all $s \in \mathbb{R}_{> 0}$.
\end{lemma}
\begin{proof}
The proof is analogous to the proof of \cite[Lem.~5.1,~p.~2575]{NVF}.
\end{proof}
\begin{lemma}
Let $L = E$ be a Galois extension of $K$. For any place $v \in \Val\left(K\right)$, let $\chi_v$ be a character of $T\left(K_v\right)$ which is non-trivial on $\mathcal{K}_v$. Then
$$
\widehat{H}_v \left(\phi_{m,v},\chi_v;-s\right) = 0.
$$
\end{lemma}
\begin{proof}
Since $\phi_{m,v}$ and $H_v$ are $\mathcal{K}_v$-invariant, the result follows by character orthogonality.
\end{proof}
\begin{corollary} \label{gfttrivcor}
Let $L = E$ be a Galois extension of $K$, and let $\chi$ be an automorphic character of $T$. If $\chi \not \in \mathcal{U}$ for $\mathcal{U}$ as in Definition \ref{kudef}, then
$$
\widehat{H}\left(\phi_m,\chi;-s\right) = 0.
$$
\end{corollary}
\begin{lemma} \label{lftlem}
Let $v \nmid \infty$ be a non-archimedean place of $K$ unramified in $L/K$, and let $\chi$ be an automorphic character of $T$ which is unramified at $v$. Then we have
\begin{equation*}
\widehat{H}_v\left(1,\chi_v;-s\right) = \left(1-\frac{1}{q_v^{ds}}\right)\prod_{w \mid v}\left(1-\frac{\chi_w\left(\pi_w\right)}{q_w^s}\right)^{-1} = L_v \left(\chi,s\right)\zeta_{K,v}\left(ds\right)^{-1}.
\end{equation*}
\end{lemma}
\begin{proof}
The result follows from \cite[Thm.~2.2.6,~p.~611]{RPBH} and Proposition \ref{fanprop}.
\end{proof}
\begin{mydef}
Given a vector $u = \left(u_1,\dots,u_r\right) \in \mathbb{N}^r$, define $f_{r,n,u}\left(x_1\dots,x_r\right)$ to be the sum of all degree-$n$ monomials in $x_1,\dots,x_r$ weighted via $u$, i.e.\
$$
f_{r,n,u}\left(x_1,\dots,x_r\right) = \sum_{\substack{\sum_{i=1}^r u_i a_i = n \\ \forall i \, a_i \in \mathbb{Z}_{\geq 0}}} x_1^{a_1}\dots x_r^{a_r}.
$$
Set
$$
f_{r,n}\left(x_1,\dots,x_r\right) = f_{r,n,\left(1,\dots,1\right)}\left(x_1,\dots,x_r\right).
$$
\end{mydef}
\begin{proposition} \label{truncprop}
Let $v \not\in S'\left(\om\right)$ be a non-archimedean place of $K$, and let $\chi \in \mathcal{U}$. Let $w_1,\dots,w_r \in \Val \left(L\right)$ be the places of $L$ over $v$. Let $u_i$ be the inertia degree of $w_i$ over $v$ for each $i=1,\dots,r$. Set
$$
c_{\chi,v,n} = f_{r,n,u}\left(\chi_{w_1}\left(\pi_{w_1}\right),\dots,\chi_{w_r}\left(\pi_{w_r}\right)\right).
$$
Then, for $\Re s > 0$, we have
$$
\widehat{H}_v\left(\phi_{m,v};\chi_v;-s\right) = 1 + \sum_{n=m}^{\infty} \frac{c_{\chi,v,n} - c_{\chi,v,n-d}}{q_v^{ns}}.
$$
\end{proposition}
\begin{proof}
Let $s \in \mathbb{C}$ with $\Re s > 0$. As $\chi \in \mathcal{U}$ and $v \not\in S'\left(\om\right)$, it follows that $\chi$ is unramified at $v$. Then, expanding geometric series, we have
\begin{equation*}
L_v\left(\chi,s\right) = \prod_{i=1}^r\left(1-\frac{\chi_{w_i}\left(\pi_{w_i}\right)}{q_v^{u_i s}}\right)^{-1} = 1+ \sum_{n=1}^{\infty}\frac{c_{\chi,v,n}}{q_v^{ns}},
\end{equation*}
so, by Lemma \ref{lftlem}, we obtain
$$
\widehat{H}_v\left(1,\chi_v;-s\right) = 1 + \sum_{n=1}^{\infty} \frac{c_{\chi,v,n} - c_{\chi,v,n-d}}{q_v^{ns}}.
$$
On the other hand, we may write
$$
\widehat{H}_v\left(1,\chi_v;-s\right) = \int_{T\left(K_v\right)}\frac{\chi_v\left(t_v\right)}{H_v\left(t_v\right)^s}d\mu_v = \sum_{n=0}^{\infty} \frac{1}{q_v^{ns}}\int_{H_v\left(t_v\right) = q_v^n}\chi_v\left(t_v\right)d\mu_v,
$$
so, comparing these expressions, we see for $n \geq 1$ that
$$
c_{\chi,v,n} - c_{\chi,v,n-d} = \int_{H_v\left(t_v\right) = q_v^n}\chi_v\left(t_v\right)d\mu_v.
$$
Since $v \not\in S'\left(\om\right)$, we have $\phi_{m,v}\left(t_v\right) = 1$ if and only if $H_v\left(t_v\right) = 1$ or $H_v\left(t_v\right) \geq q_v^m$, so the result follows.
\end{proof}
\subsection{Regularisation}
Now that we have expressions for the local Fourier transforms at all but finitely many places, our goal is to find ``regularisations'' for the global Fourier transforms, i.e.\ functions expressible as Euler products whose convergence is well-understood and whose local factors approximate the local Fourier transforms well (as expansions in $q_v$) at all but finitely many places. As in \cite{RPBH}, \cite{NVF} and \cite{CPBH}, we will construct our regularisations from $L$-functions.

\begin{proposition} \label{partprop}
Let $G$ be a subgroup of $S_d$ acting freely and transitively on $\{1,\dots,d\}$, and let $m \geq 2$ be a positive integer. Let $S_m$ act upon $G^m$ by permutation of coordinates, and let $G$ act on $G^m/S_m$ by right multiplication of every element of a representative $m$-tuple. Set $S\left(G,m\right) = \left(G^m/S_m\right)/G$.
\begin{enumerate}[label=(\roman*)]

\item If $\overline{\left(g_1,\dots,g_m\right)} = \overline{\left(h_1,\dots,h_m\right)}$ in $S\left(G,m\right)$, then
\begin{equation} \label{sumeq}
\sum_{i=1}^d x_{g_1\left(i\right)}\dotsm x_{g_m\left(i\right)} = \sum_{i=1}^d x_{h_1\left(i\right)}\dotsm x_{h_m\left(i\right)},
\end{equation}
hence, for $\overline{\left(g_1,\dots,g_m\right)} \in S\left(G,m\right)$, we may define the sum
\[
\phi_{\overline{\left(g_1,\dots,g_m\right)}}\left(x_1,\dots,x_d\right) = \sum_{i=1}^d x_{g_1\left(i\right)}\dotsm x_{g_m\left(i\right)}.
\]

\item If $\overline{\left(g_1,\dots,g_m\right)} \neq \overline{\left(h_1,\dots,h_m\right)}$, then the sums $\phi_{\overline{\left(g_1,\dots,g_m\right)}}\left(x_1,\dots,x_d\right)$ and $\phi_{\overline{\left(h_1,\dots,h_m\right)}}\left(x_1,\dots,x_d\right)$ share no summands. Further, a monomial appears twice in $\phi_{\overline{\left(g_1,\dots,g_m\right)}}\left(x_1,\dots,x_d\right)$ if and only if $\left(g_1,\dots,g_m\right) = \left(g_1 g_r,\dots,g_m g_r\right)$ for some $r \in \{1,\dots,m\}$ with $g_r \neq 1_G$.

\item If $d$ is coprime to $m$, then we have
\begin{equation*}
f_{d,m}\left(x_1,\dots,x_d\right) = \sum_{\overline{\left(g_1,\dots,g_m\right)} \in S\left(G,m\right)} \phi_{\left(g_1,\dots,g_m\right)}\left(x_1,\dots,x_d\right).
\end{equation*}

\item If $d$ is prime and $m = kd$, then we have
\begin{equation*}
f_{d,m}\left(x_1,\dots,x_d\right) + \left(d-1\right)x_1^k \dotsm x_d^k = \sum_{\overline{\left(g_1,\dots,g_m\right)} \in S\left(G,m\right)} \phi_{\left(g_1,\dots,g_m\right)}\left(x_1,\dots,x_d\right).
\end{equation*}
\end{enumerate}
\end{proposition}
\begin{proof}

First, we prove (i). 
Note that $\overline{\left(g_1,\dots,g_m\right)} = \overline{\left(h_1,\dots,h_m\right)}$ if and only if $\{h_1,\dots,h_m\} = \{g_1g,\dots,g_m g\}$ as multisets for some $g \in G$.
If the coordinates of $\left(h_1,\dots,h_m\right) \in G^m$ are a permutation of those of $\left(g_1,\dots,g_m\right) \in G^m$, then
$$
x_{h_1\left(i\right)}\dotsm x_{h_m\left(i\right)} = x_{g_1\left(i\right)}\dotsm x_{g_m\left(i\right)}
$$
for any $i \in \{1,\dots,d\}$, while if $\left(h_1,\dots,h_m\right) = \left(g_1g,\dots,g_m g\right)$ for some $g \in G$, then for any $i \in \{1,\dots,d\}$, we have
$$
x_{g_1 g\left(i\right)}\dotsm x_{g_m g\left(i\right)} = x_{g_1\left(j\right)}\dotsm x_{g_m\left(j\right)}
$$
for the unique $j$ with $g\left(i\right) = j$. In either case, we obtain \eqref{sumeq}. The claim follows.

Next we prove (ii).
Note that for $\overline{\left(g_1,g_2,\dots,g_m\right)} \in S\left(G,m\right)$, we may take $g_1 = 1_G$ without loss of generality. Suppose that
$$
x_{i}x_{g_2\left(i\right)}\dotsm x_{g_m\left(i\right)} = x_{j} x_{h_2\left(j\right)}\dotsm x_{h_m\left(j\right)}
$$
for some $i,j \in \{1,\dots,d\}$. This is equivalent to the equality of multisets
$$
\{i,g_2\left(i\right),\dots,g_m\left(i\right)\} = \{j,h_2\left(j\right),\dots,h_m\left(j\right)\}.
$$
If $i=j$, we have
$$
\{g_2\left(i\right),\dots,g_m\left(i\right)\} = \{h_2\left(i\right),\dots,h_m\left(i\right)\},
$$
and by the freeness of the action of $G$ on $\{1,\dots, d\}$, we have $\left(g_2,\dots,g_m\right) = \left(h_2,\dots,h_m\right)$ up to reordering, i.e.\ the $m$-tuple $\left(1_G,h_2,\dots,h_m\right)$ is a permutation of $\left(1_G,g_2,\dots,g_m\right)$. If $i \neq j$, we may take $g_2\left(i\right) = j$ without loss of generality (note that $g_2 \neq 1_G$ in this case), and we obtain the equality of multisets
$$
\{j, h_2\left(j\right),\dots,h_m\left(j\right)\} = \{g_2\left(i\right), h_2 g_2 \left(i\right),\dots,h_m g_2\left(i\right)\}.
$$
Once again, by freeness of the action of $G$ on $\{1,\dots,d\}$, we get that
$$
\{1_G,g_2,\dots,g_m\} = \{g_2,h_2 g_2, \dots, h_m g_2\}
$$
as multisets, from which both claims follow.

Before proceeding to the proofs of (iii) and (iv), we make the following observation: if a multiset $M = \{g_1,\dots,g_m\}$ of elements of $G$ is closed under right multiplication by some $g \in G$ of order $n \geq 2$, then $n$ divides $m$. Indeed, $M$ must contain the $n$ distinct elements $g_1, g_1 g, g_1 g^2,\dots,g_1 g^{n-1}$. Without loss of generality, we may take $g_l = g_1 g^{l-1}$ for $l =1,\dots,n$. Since $\{g_1,\dots,g_n\} = \{g_1,g_1 g,\dots, g_1 g^{n-1}\}$ is clearly closed under right multiplication by $g$, so is the sub-multiset $\{g_{n+1},\dots,g_m\}$. Iterating this process, we deduce that $n$ divides the cardinality of $M$, as otherwise we would eventually obtain a non-empty sub-multiset closed under right multiplication by $g$ with fewer than $n$ elements, which is clearly impossible.

We now prove (iii) and (iv). By (i) and (ii), every degree-$m$ monomial in $x_1,\dots,x_d$ appears in a unique summand of $\sum_{\overline{\left(g_1,\dots,g_m\right)} \in S\left(G,m\right)} \phi_{\left(g_1,\dots,g_m\right)}\left(x_1,\dots,x_d\right)$, and a monomial appears twice in $\phi_{\overline{\left(g_1,\dots,g_m\right)}}\left(x_1,\dots,x_d\right)$ if and only if the multiset $\{g_1,\dots,g_m\}$ is closed under right multiplication by some $g_r \neq 1_G$. By the above, the latter is possible only if the order of $g_r$ divides $m$, while by Lagrange's theorem, the order of $g_r$ divides $d$. If $d$ is coprime to $m$, then we deduce that no such $g_r$ exists, thus we obtain (iii). If $d=p$ is prime and $m=kp$, then $g_r$ is necessarily a generator of the cyclic group $G$. It then follows as in the above observation that
\begin{equation*}
\overline{\left(g_1,\dots,g_{kp}\right)} = \overline{\left(1_G,g_r,\dots,g_r^{p-1},\dots,1_G,g_r,\dots,g_r^{p-1}\right)}.
\end{equation*}
Letting $g$ be a generator of $G$, we see that $\phi_{\overline{\left(1_G,g,\dots,g^{p-1},\dots,1_G,g,\dots,g^{p-1}\right)}}\left(x_1,\dots,x_p\right) = px_1^k \dotsm x_p^k$ is the only one of the polynomials $\phi_{\overline{\left(g_1,\dots,g_{kp}\right)}}\left(x_1,\dots,x_p\right)$, $\overline{\left(g_1,\dots,g_{kp}\right)} \in S\left(G,kp\right)$ in which a monomial appears more than once, and so we obtain (iv).
\end{proof}
\begin{remark} \label{sgmrem}
It follows immediately from Proposition \ref{partprop} that
$$
\#S\left(G,m\right) = 
\begin{cases}
\frac{1}{d}{{d+m-1}\choose{d-1}} \textrm{ if $d$ and $m$ are coprime}, \\
\frac{1}{d}\left({{d+m-1}\choose{d-1}} - 1\right) + 1 \textrm{ if $d$ is prime and $d$ divides $m$,}
\end{cases}
$$
since the number of monomials of degree $m$ in $d$ variables is ${{d+m-1}\choose{d-1}}$.
\end{remark}
\begin{remark}
As we shall shortly see, it is the ability to partition (or in the case where $d$ is a prime dividing $m$, nearly partition) the sum of all degree-$m$ monomials in $d$ variables as in Proposition \ref{partprop} that allows us to construct well-behaved regularisations.
\end{remark}
For the rest of this section, let $L = E$ be Galois over $K$ with Galois group $G$, and assume that $m$ is coprime to $d$ if $d$ is not prime.
\begin{lemma} \label{canclem}
Let $v \not\in S'\left(\om\right)$ be a non-archimedean place which is totally split in $E/K$ and let $\chi \in \mathcal{U}$. Then
$$
\frac{\prod_{\overline{\left(g_1,\dots,g_m\right)} \in S\left(G,m\right)} L_v\left(\chi^{g_1}\dotsm\chi^{g_m},ms\right)}{\prod_{\overline{\left(h_1,\dots,h_{m-d}\right)} \in S\left(G,m-d\right)} L_v\left(\chi^{h_1}\dotsm\chi^{h_{m-d}},ms\right)} = \prod_{\overline{\left(g_1,\dots,g_m\right)} \in S'\left(G,m\right)} L_v\left(\chi^{g_1}\dotsm\chi^{g_m},ms\right),
$$
where
$$
S'\left(G,m\right) = \left\{\overline{\left(g_1,\dots,g_m\right)} \in S\left(G,m\right) : \#\{g_1,\dots,g_m\} \leq d-1\right\}.
$$.
\end{lemma}
\begin{proof}
Let $G=\{g_1,\dots,g_d\}$. First, we show that every factor of the denominator of the left-hand side appears on the numerator. Let $\overline{\left(h_1,\dots,h_{m-d}\right)} \in S\left(G,m-d\right)$. Then we claim that
$$
L_v\left(\chi^{h_1} \dotsm \chi^{h_{m-d}},ms\right) = L_v\left(\chi^{h_1} \dotsm \chi^{h_{m-d}} \chi^{g_1} \dotsm \chi^{g_d},ms\right).
$$
Since $G$ acts freely and transitively on the places $w_1,\dots,w_d$ of $E$ over $v$, we have
$$
\{\chi_{w_i}^{g_1}\left(\pi_{w_i}\right),\dots,\chi_{w_i}^{g_d}\left(\pi_{w_i}\right)\} = \{\chi_{w_1}\left(\pi_{w_1}\right),\dots,\chi_{w_d}\left(\pi_{w_d}\right)\}
$$
for any $i=1,\dots,d$. Since $\chi_v$ is trivial on $\mathcal{K}_v = T\left(\mathcal{O}_v\right)$, we have from Note \ref{torhecnote} that $\prod_{i=1}^d\chi_{w_i}\left(\pi_{w_i}\right) = \chi_v\left(1\right) = 1$, so $\left(\chi^{g_1}\dotsm\chi^{g_d}\right)_w\left(\pi_w\right) =1$ for all $w \mid v$. Then the equality follows.

It now suffices to show that, for $\overline{\left(h_1,\dots,h_{m-d}\right)} \neq \overline{\left(h'_1,\dots,h'_{m-d}\right)}$, we have
$$
\overline{\left(h_1,\dots,h_{m-d},g_1,\dots,g_d\right)} \neq \overline{\left(h'_1,\dots,h'_{m-d},g_1,\dots,g_d\right)}.
$$
If not, then  $\{h'_1,\dots,h'_{m-d},g_1,\dots,g_d\} = \{gh_1,\dots,gh_{m-d},gg_1,\dots,gg_d\}$ as multisets for some $g \in G$. Since we have $\{gg_1,\dots,gg_d\} = \{g_1,\dots,g_d\}$, this implies that $\{h'_1,\dots,h'_{m-d}\} = \{gh_1,\dots,gh_{m-d}\}$ as multisets, but then we have $\overline{\left(h_1,\dots,h_{m-d}\right)} = \overline{\left(h'_1,\dots,h'_{m-d}\right)}$, which is false.
\end{proof}
\begin{remark} \label{sgmrem2}
It follows from the proof of Lemma \ref{canclem} that $\#S'\left(G,m\right) = S\left(G,m\right) - S\left(G,m-d\right)$. Combining this with Remark \ref{sgmrem}, we obtain
$$
\#S'\left(G,m\right) = \frac{1}{d}\left({{d+m-1}\choose{d-1}} - {{m-1}\choose{d-1}}\right) = b\left(d,m\right).
$$
Note that the term $\frac{1}{d}{{m-1}\choose{d-1}}$ only appears when $d \leq m$.
\end{remark}
\begin{mydef} \label{pseriesdef}
For all $\chi \in \mathcal{U}$, $\Re s > 0$ and non-archimedean places $v \nmid \infty$, set
$$
F_{m,\chi,v}\left(s\right) = \prod_{\overline{\left(g_1,\dots,g_m\right)} \in S'\left(G,m\right)} L_v\left(\chi^{g_1}\dotsm\chi^{g_m},ms\right), \quad
G_{m,\chi,v}\left(s\right) = \frac{\widehat{H}_v\left(\phi_{m,v},\chi_v;-s\right)}{F_{m,\chi,v}\left(s\right)},
$$
and define
\begin{equation} \label{regeq}
F_{m,\chi}\left(s\right) = \prod_{v \nmid \infty} F_{m,\chi,v}\left(s\right) = \prod_{\overline{\left(g_1,\dots,g_m\right)} \in S'\left(G,m\right)} L\left(\chi^{g_1}\dotsm\chi^{g_m},ms\right),
\end{equation}
$$
G_{m,\chi}\left(s\right) = \prod_{v \nmid \infty} G_{m,\chi,v}\left(s\right).
$$
For any non-archimedean place $v \nmid \infty$, write
$$
\widehat{H}_v\left(\phi_{m,v},\chi_v;-s\right) = \sum_{n=0}^{\infty}\frac{a_{\chi,v,n}}{q_v^{ns}},
$$
where $a_{\chi,v,n} = \int_{H_v\left(t_v\right) = q_v^n}\phi_{m,v}\left(t_v\right)\chi_v\left(t_v\right)d\mu_v$, and write
$$
F_{m,\chi,v}\left(s\right) = 1 + \sum_{n=1}^{\infty}\frac{b_{\chi,v,mn}}{q_v^{mns}}
$$
for the expansion of $F_{m,\chi,v}\left(s\right)$ as a multidimensional geometric series in $q_v^{ms}$, so
$$
G_{m,\chi,v}\left(s\right) = \frac{\sum_{n=0}^{\infty}\frac{a_{\chi,v,n}}{q_v^{ns}}}{ 1 + \sum_{n=1}^{\infty}\frac{b_{\chi,v,mn}}{q_v^{mns}}} = \sum_{n=0}^{\infty} \frac{d_{\chi,v,n}}{q_v^{ns}},
$$
where $d_{\chi,v,n}$ is defined for all $n \geq 0$ by the iterative formula
\begin{equation} \label{deq}
d_{\chi,v,n} = a_{\chi,v,n} - \sum_{r=1}^{\lfloor \frac{n}{m} \rfloor} b_{\chi,v,mr}d_{\chi,v,n-mr}.
\end{equation}
In particular, we have $d_{\chi,v,n} = a_{\chi,v,n}$ for $0 \leq n \leq m-1$.
\end{mydef}
\begin{corollary} \label{lftcor}
For $v \not\in S'\left(\om\right)$ a non-archimedean place, we have $d_{\chi,v,0} = 1$ and $d_{\chi,v,n} = 0$ for all $n \in \{1,\dots,m\}$.
\end{corollary}
\begin{proof}
Let $c_{\chi,v,n}$ be as defined in Proposition \ref{truncprop}. Since $v \not\in S'\left(\om\right)$, we have from \emph{loc.\ cit.\ }that $a_{\chi,v,0} =1$, $a_{\chi,v,n} = 0$ for $n \in \{1,\dots,m-1\}$ and $a_{\chi,v,m} = c_{\chi,v,m} - c_{\chi,v,m-d}$. Then, by \eqref{deq}, we see that $d_{\chi,v,0} = a_{\chi,v,0} = 1$ and $d_{\chi,v,n} = a_{\chi,v,n} = 0$ for $1 \leq n \leq m-1$. Further, we obtain
$$
d_{\chi,v,m} = a_{\chi,v,m} - b_{\chi,v,m}d_{\chi,v,0} = c_{\chi,v,m} - c_{\chi,v,m-d} - b_{\chi,v,m},
$$
so, to complete the proof, it suffices to show that $b_{\chi,v,m} = c_{\chi,v,m} - c_{\chi,v,m-d}$. \\
Since $E/K$ is Galois, all of the places $w_1,\dots,w_r$ of $E$ over $v$ share a common inertia degree $d_v$. Since $\chi_v\left(T\left(\mathcal{O}_v\right)\right) = 1$, it is unramified as a Hecke character at all of the $w_i$ (see Note \ref{torhecnote}), and for any $g_1,\dots,g_m \in G$, so is $\chi^{g_1}\dotsm\chi^{g_m}$. Then
\begin{equation} \label{lfneq}
\begin{aligned}
L_v\left(\chi^{g_1} \dotsm \chi^{g_m},ms\right) & = \prod_{i=1}^r\left(1-\frac{\left(\chi^{g_1}\dotsm\chi^{g_m}\right)_{w_i}\left(\pi_{w_i}\right)}{q_v^{d_v ms}}\right)^{-1}  \\
&  = 1 + \frac{1}{q_v^{d_v ms}} \sum_{i=1}^r \left(\chi^{g_1}\dotsm\chi^{g_m}\right)_{w_i} \left(\pi_{w_i}\right) + O\left(\frac{1}{q_v^{\left(d_v m+1\right)s}}\right).
\end{aligned}
\end{equation}

First, suppose that $v$ is totally split in $E/K$. Then \eqref{lfneq} gives
$$
L_v\left(\chi^{g_1} \dotsm \chi^{g_m},ms\right) = 1 + \frac{ \phi_{\left(g_1,\dots,g_m\right)} \left(\chi_{w_1}\left(\pi_{w_1}\right),\dots,\chi_{w_d}\left(\pi_{w_d}\right)\right)}{q_v^{ms}} + O\left(\frac{1}{q_v^{\left(m+1\right)s}}\right).
$$
Since $G$ acts freely and transitively on the $w_i$, it follows from Proposition \ref{partprop} and Lemma \ref{canclem} that $b_{\chi,v,m} = c_{\chi,v,m} - c_{\chi,v,m-d}$, and so $d_{\chi,v,m} = 0$.

Now assume that $v$ is not totally split in $E/K$. If $\gcd \left(d,m\right) = 1$, then $c_{\chi,v,m} = c_{\chi,v,m-d}  = 0$, as $c_{\chi,v,n} = 0$ whenever $d_v \nmid n$. If $d$ is prime, then $v$ is inert and we have $\widehat{H}_v\left(\phi_{m,v},\chi_v;-s\right) = 1$ since $T\left(K_v\right) = T\left(\mathcal{O}_v\right)$. Then, in either case, we have $c_{\chi,v,m} - c_{\chi,v,m-d} = 0$, and \eqref{lfneq} implies that $b_{\chi,v,m} = 0$, hence $d_{\chi,v,m} = 0$.
\end{proof}
\begin{corollary} \label{gftcor}
For any $\chi \in \mathcal{U}$, we have
$$
\widehat{H}\left(\phi_{m};\chi;-s\right) = \prod_{v \mid \infty}\widehat{H}_v\left(1,\chi_v;-s\right)F_{m,\chi}\left(s\right) G_{m,\chi}\left(s\right),
$$
where $G_{m,\chi}\left(s\right)$ is holomorphic and uniformly bounded with respect to $\chi$ for $\Re s \geq \frac{1}{m}$ and $G_{m,1}\left(\frac{1}{m}\right) \neq 0$. In particular, $\widehat{H} \left(\phi_m,\chi;-s\right)$ possesses a holomorphic continuation to the line $\Re s = \frac{1}{m}$, apart from possibly at $s = \frac{1}{m}$. When $\chi=1$, the right-hand side has a pole of order $b\left(d,m\right)$ at $s=\frac{1}{m}$.
\end{corollary}
\begin{proof}
By construction, $\widehat{H}\left(\phi_{m};\chi;-s\right) = \prod_{v \mid \infty}\widehat{H}_v\left(1,\chi_v;-s\right)F_{m,\chi}\left(s\right) G_{m,\chi}\left(s\right)$. We will prove the stronger result that $G_{m,\chi}\left(s\right)$ is holomorphic on $\Re s > \frac{1}{m+1}$ and uniformly bounded with respect to both $\chi$ and $\epsilon$ on $\Re s \geq \frac{1}{m+1}+\epsilon$ for all $\epsilon > 0$.

For a place $v \nmid \infty$ and $s \in \mathbb{C}$ with $\Re s = \sigma \geq \epsilon$ for some $\epsilon > 0$, we have
\begin{equation*}
\begin{aligned}
\sum_{n=0}^{\infty}\left|\frac{a_{\chi,v,n}}{q_v^{ns}}\right| = \sum_{n=0}^{\infty}\frac{1}{q_v^{n\sigma}}\left|\int_{H_v\left(t_v\right)=q_v^n}\phi_{m,v}\left(t_v\right)\chi_v\left(t_v\right)d\mu_v\right| & \\
\leq \sum_{n=0}^{\infty}\frac{1}{q_v^{n\sigma}}\int_{H_v\left(t_v\right)=q_v^n}\left|\phi_{m,v}\left(t_v\right)\chi_v\left(t_v\right)\right|d\mu_v & \\
= \int_{T\left(K_v\right)}\left|\frac{\phi_{m,v}\left(t_v\right)\chi_v\left(t_v\right)}{H_v\left(t_v\right)^s}\right|d\mu_v,\\
\end{aligned}
\end{equation*}
so, by Lemma \ref{aclem}, the series $\sum_{n=0}^{\infty}\frac{a_{\chi,v,n}}{q_v^{ns}}$ is absolutely convergent and bounded by a constant depending only on $\epsilon$ and $v$. Now, for any $N \in \mathbb{N}$, we have
$$
\left|\sum_{n=0}^{\infty}\frac{a_{\chi,v,n}}{q_v^{ns}} - \sum_{n=0}^{N}\frac{a_{\chi,v,n}}{q_v^{ns}}\right| = \left|\sum_{n=N+1}^{\infty}\frac{a_{\chi,v,n}}{q_v^{ns}}\right| \leq \sum_{n=0}^{\infty}\frac{|a_{\chi,v,n}|}{q_v^{n\epsilon}},
$$
from which it follows that $\sum_{n=0}^{\infty}\frac{a_{\chi,v,n}}{q_v^{ns}}$ is also uniformly convergent, hence the function $\widehat{H}_v\left(\phi_{m,v},\chi_v;-s\right)$ is holomorphic on $\Re s > 0$. Then, we note that $F_{m,\chi,v}\left(s\right)$ is clearly holomorphic on $\Re s > 0$, and we have
$$
\frac{1}{|F_{m,\chi,v}\left(s\right)|} = \prod_{\overline{\left(g_1,\dots,g_m\right)} \in S'\left(G,m\right)} \left|L_v\left(\chi^{g_1}\dotsm\chi^{g_m},ms\right)^{-1}\right| \leq \left(1 + \frac{1}{q_v^{m\epsilon}}\right)^{db\left(d,m\right)},
$$
hence $G_{m,\chi,v}\left(s\right)$ is holomorphic on $\Re s > 0$ and is bounded uniformly in terms of $\epsilon$ and $v$ on $\Re s \geq \epsilon$.

To conclude the result, it suffices to prove that there exists $N \in \mathbb{N}$ such that
$$
\prod_{q_v > N }G_{m,\chi,v}\left(s\right)
$$
is holomorphic and uniformly bounded with respect to $\chi$ on $\Re s \geq \frac{1}{m+1} + \epsilon$ for all $\epsilon > 0$.
Let $v \not\in S'\left(\om\right)$ be non-archimedean, and let $\Re s = \sigma \geq \frac{1}{m+1} + \epsilon$. From
$$
\widehat{H}_v\left(\phi_{m,v},\chi_v;-s\right) = \left(1-\frac{1}{q_v^{ds}}\right)L_v\left(\chi,s\right),
$$
and the definition of $F_{m,\chi,v}\left(s\right)$, we have
$$
|a_{\chi,v,n}| \leq 2d^n, \quad |b_{\chi,v,n}| \leq \left(b\left(d,m\right)d\right)^n.
$$
Then, by \eqref{deq}, it follows inductively that we have
\begin{equation} \label{dineq}
|d_{\chi,v,n}| \leq 2^n \left(b\left(d,m\right)d\right)^n = \left(2b\left(d,m\right) d\right)^n.
\end{equation}
Choose $N > \left(2b\left(d,m\right)d\right)^{\frac{1}{\sigma}}$ so that, for all places $v \nmid \infty$ with $q_v > N$, we have $v \not\in S'\left(\om\right)$.
Now, any normally convergent infinite product is holomorphic (see \cite[\S2]{CFT}), and $\prod_{q_v > N}G_{m,\chi,v}\left(s\right)$ converges normally if and only if
$$
\sum_{q_v > N} \sum_{n=m+1}^{\infty}\frac{|d_{\chi,v,n}|}{q_v^{n\sigma}}
$$
converges. By \eqref{dineq} and our condition on $N$, we need only check convergence of
$$
\sum_{q_v > N}\frac{1}{q_v^{\left(m+1\right)\sigma}},
$$
which is clear.
Then $G_{m,\chi}\left(s\right)$ is holomorphic on $\Re s > \frac{1}{m+1}$. Further, for $\Re s \geq \frac{1}{m+1}+\epsilon$, we have the bound
$$
\left|\prod_{q_v > N}G_{m,\chi,v}\left(s\right)\right| \leq \prod_{q_v > N}\left(1+\sum_{n=m+1}^{\infty}\left(\frac{2b\left(d,m\right)d}{q_v^{\frac{1}{m+1}+\epsilon}}\right)^n\right),
$$
which is uniform with respect to $\chi$.
Now, as a convergent infinite product, $G_{m,1}\left(\frac{1}{m}\right)$ is zero if and only if $G_{m,1,v}\left(\frac{1}{m}\right) = \frac{\widehat{H}_v\left(\phi_{m,v},1;-\frac{1}{m}\right)}{F_{m,1,v}\left(\frac{1}{m}\right)} = 0$ for some place $v \nmid \infty$. However, $\widehat{H}_v\left(\phi_{m,v},1;-\frac{1}{m}\right) \neq 0$ by Lemma \ref{nontrivlem}, so we conclude that $G_{m,1}\left(\frac{1}{m}\right) \neq 0$.
The order of the pole of the right-hand side being $b\left(d,m\right)$ follows from Theorem \ref{hlfthm}, since
\[
F_{m,1}\left(s\right) = \zeta_E\left(ms\right)^{b\left(d,m\right)}. \qedhere
\]
\end{proof}
\begin{note}
In constructing the regularisation $F_{m,\chi}\left(s\right)$, one must ensure that
$$
\frac{\widehat{H}_v\left(\phi_{m,v},\chi_v;-s\right)}{F_{m,\chi,v}\left(s\right)} = 1 + O\left(\frac{1}{q_v^{\left(m+1\right)s}}\right)
$$
for all non-archimedean places $v$ with $q_v$ is sufficiently large. As seen above, the restrictions on $d$, $m$ and $E$ ensure that this is automatic for all such places which are not totally split, i.e.\ we only need to approximate the local Fourier transform at totally split places not in $S'\left(\om\right)$. Without these restrictions, one might have to approximate the local Fourier transform at places of more than one splitting type simultaneously, and to do this would require a new approach.
\end{note}
Before applying our key theorems, we give one more result, which will be used in order to move from the Poisson summation formula to the Tauberian theorem.
\begin{lemma} \label{latsumlem} \cite[Lem.~5.9,~p.~2577]{NVF}
Choose an $\mathbb{R}$-vector space norm $\|\cdot\|$ on $X^*\left(T_{\infty}\right)_{\mathbb{R}}$ and let $\mathcal{L} \subset X^*\left(T_{\infty}\right)_{\mathbb{R}}$ be a lattice. Let $C$ be a compact subset of $\Re s \geq \frac{1}{m}$ and let $g: X^*\left(T_{\infty}\right)_{\mathbb{R}} \times C \rightarrow \mathbb{C}$ be a function. If there exists $0 \leq \delta < \frac{1}{d-1}$ such that
$$
|g\left(\psi,s\right)| \ll_C \left(1+ \|\psi\|\right)^\delta
$$  
for all $\psi \in X^*\left(T_{\infty}\right)_{\mathbb{R}}$ and $s \in C$, then the sum
$$
\sum_{\psi \in \mathcal{L}}g\left(\psi,s\right)\prod_{v \mid \infty}\widehat{H}_{v}\left(1,\psi;-s\right)
$$
is absolutely and uniformly convergent on $C$.
\end{lemma}
\begin{theorem} \label{zetathm}
Let
$$
\Omega_{m}\left(s\right) = Z_m \left(s\right) \left(s-\frac{1}{m}\right)^{b\left(d,m\right)}.
$$
Then $\Omega_{m}\left(s\right)$ admits an extension to a holomorphic function on $\Re s \geq \frac{1}{m}$.
\end{theorem}
\begin{proof}
Let $s \in \mathbb{C}$ with $\Re s > \frac{1}{m}$. Combining the formal application \eqref{formappeq} of the Poisson summation formula with Lemma \ref{vollem} and Corollary \ref{gfttrivcor} gives
\begin{equation} \label{psfeq}
\begin{aligned}
Z_m \left(s\right) = \frac{\Res_{s=1}\zeta_K\left(s\right)}{d \Res_{s=1}\zeta_E\left(s\right)} \sum_{\chi \in \mathcal{U}} \widehat{H}\left(\phi_m,\chi;-s\right).
\end{aligned}
\end{equation}
By Corollary \ref{gftcor}, the function $s \mapsto \frac{\phi_m\left(t\right)}{H\left(t\right)^s}$ is $L^1$ for $\Re s > \frac{1}{m}$. To show that this application is valid, we apply Bourqui's criterion \cite[Cor.~3.36,~p.~64]{FZH}, by which it suffices to show that the right-hand side of \eqref{psfeq} is absolutely convergent, $s \mapsto \frac{\phi_m\left(t\right)}{H\left(t\right)^s}$ is continuous and there exists an open neighbourhood $U \subset T\left(\mathbb{A}_K\right)$ of the origin and strictly positive constants $C_1$ and $C_2$ such that for all $u \in U$ and all $t \in T\left(\mathbb{A}_K\right)$, we have
$$
C_1 \left|\frac{\phi_{m}\left(t\right)}{H\left(t\right)^s}\right| \leq \left|\frac{\phi_{m}\left(u t\right)}{H\left(u t\right)^s}\right| \leq C_2 \left|\frac{\phi_{m}\left(t\right)}{H\left(t\right)^s}\right|.
$$
We may take $U = \mathcal{K}$ by Lemma \ref{mcslem}, and continuity is clear. It only remains to prove the absolute convergence. We will prove the stronger result that
$$
\sum_{\chi \in \mathcal{U}} \widehat{H}\left(\phi_m,\chi;-s\right)\left(s-\frac{1}{m}\right)^{b\left(d,m\right)}
$$
is absolutely and uniformly convergent on any compact subset $C$ of the half-plane $\Re s \geq \frac{1}{m}$, which will both verify validity of the application and prove the theorem.

Since $\mathcal{K} \subset \mathsf{K}_T$ is of finite index, the map \eqref{taieq} yields a homomorphism
$$
\mathcal{U} \rightarrow X^*\left(T_{\infty}\right)_{\mathbb{R}}, \quad
\chi \mapsto \chi_{\infty},
$$
with finite kernel $\mathcal{N}$ and image $\mathcal{L}$ a lattice of full rank. We obtain
\[
\begin{aligned}
& \sum_{\chi \in \mathcal{U}} \widehat{H}\left(\phi_m,\chi;-s\right)\left(s-\frac{1}{m}\right)^{b\left(d,m\right)} \\
& = \sum_{\psi \in \mathcal{L}} \prod_{v \mid \infty} \widehat{H}_{v}\left(1,\psi;-s\right) \sum_{\substack{\chi \in \mathcal{U} \\ \chi_{\infty} = \psi}} \prod_{v \nmid \infty} \widehat{H}_v\left(\phi_{m,v},\chi_v;-s\right) \left(s-\frac{1}{m}\right)^{b\left(d,m\right)},
\end{aligned}
\]
where the inner sum is finite. Then, for $s \in C$, we have
\[
\begin{aligned}
& \sum_{\chi \in \mathcal{U}} \widehat{H}\left(\phi_m,\chi;-s\right)\left(s-\frac{1}{m}\right)^{b\left(d,m\right)} \\
& \ll \sum_{\psi \in \mathcal{L}} \prod_{v \mid \infty} \left|\widehat{H}_v\left(1,\psi;-s\right)\right| \sum_{\substack{\chi \in \mathcal{U} \\ \chi_{\infty} = \psi}} \prod_{v \nmid \infty} \left|\widehat{H}_v\left(\phi_{m,v},\chi_v;-s\right) \left(s-\frac{1}{m}\right)^{b\left(d,m\right)}\right|.
\end{aligned}
\]
Now, for $\chi \in \mathcal{U}$, we deduce from the proof of Corollary \ref{gftcor} that
\[
\left|\prod_{v \nmid \infty} \widehat{H}_v\left(\phi_{m,v},\chi_v;-s\right) \left(s-\frac{1}{m}\right)^{b\left(d,m\right)}\right|
\ll_C \left|F_{m,\chi}\left(s\right)\left(s-\frac{1}{m}\right)^{b\left(d,m\right)}\right|.
\]
In order to deduce the result from Lemma \ref{latsumlem}, it suffices to prove that, for each $\psi \in \mathcal{L}$  and some constant $0 \leq \delta < \frac{1}{d-1}$, we have
$$
\left|\sum_{\substack{\chi \in \mathcal{U} \\ \chi_{\infty} = \psi}}F_{m,\chi}\left(s\right)\left(s-\frac{1}{m}\right)^{b\left(d,m\right)}\right| \ll_C \left(1+ \|\psi\|\right)^{\delta}
$$
for $\|\cdot\|$ as in Definition \ref{vecnormdef}. As $\mathcal{K} \subset \mathsf{K}_T$ is of finite index, there exists a constant $Q>0$ such that $q\left(\chi\right) < Q$  for all $\chi \in \mathcal{U}$ (cf.\ \cite[Proof~of~Thm.~5.12,~p.~2579]{NVF}). 
Since $F_{m,\chi}\left(s\right)$ is a product of $b\left(d,m\right)$ $L$-functions of Hecke characters evaluated at $ms$, it follows from Lemma \ref{lflem} that
$$
\left|\sum_{\substack{\chi \in \mathcal{U} \\ \chi_{\infty} = \psi}}F_{m,\chi}\left(s\right)\left(s-\frac{1}{m}\right)^{b\left(d,m\right)}\right| \ll_{\varepsilon,C}|\mathcal{N}| \cdot Q^{\varepsilon}\left(1+\|\psi\|\right)^{\varepsilon}
$$
for all for all $\varepsilon > 0$ and $s \in C$. The result now follows from Lemma \ref{latsumlem}.
\end{proof}
%%
%%%%%%%%%%%%%%%%%%%%%%%%%%%%%%%%%%%%%%%%%%%%%%%%%%%%%%%%%%%%%%%%
\subsection{The leading constant}
In order to apply \cite[Thm.~3.3.2,~p.~624]{RPBH} and deduce Theorem \ref{mainthm} from Theorem \ref{zetathm}, it only remains to show that $\Omega_m\left(\frac{1}{m}\right) \neq 0$.
\begin{mydef}
Let $\mathcal{U}^G$ be the subgroup of $G$-invariant elements of $\mathcal{U}$, and set
$$
\mathcal{U}_0 = 
\begin{cases}
\mathcal{U}[m] \textrm{ if } d=2, \\
\mathcal{U}^G \cap \mathcal{U}[m] \textrm{ otherwise.}
\end{cases}
$$
\end{mydef}
\begin{lemma}
For any Galois extension of number fields $E/K$, the subgroup $\mathcal{U}[m] \leq \left(T\left(\mathbb{A}_K\right)/T\left(K\right)\right)^\wedge$ is finite. In particular, $\mathcal{U}_0$ is a finite subgroup of $\mathcal{U}$.
\end{lemma}
\begin{proof}
By class field theory \cite[Ch.~VI,~\S6;~Ch.~VII,~\S6]{ANTN}, $\mathcal{U}$ may be interpreted as a subset of $\Gal\left(E^{\ab}_{S'\left(\om\right)}/E\right)^\wedge$ for $E^{\ab}_{S'\left(\om\right)}$ the maximal $S'\left(\om\right)$-unramified abelian extension of $E$, hence $\mathcal{U}[m]$ is in bijection with a subset of $\Hom\left(\Gal\left(E^{\textrm{ab}}_{S'\left(\om\right)}/E\right),\mu_m\right)$, a finite set.
\end{proof}
\newpage
\begin{lemma} \label{contriblem} \leavevmode
\begin{enumerate}[label=(\roman*)]
\item The characters $\chi \in \mathcal{U}$ contributing to the pole of $Z_m\left(s\right)$ of order $b\left(d,m\right)$ at $s=\frac{1}{m}$ are precisely those $\chi \in \mathcal{U}_0$ such that $\prod_{v \mid \infty}\widehat{H}_v\left(1,\chi_v;-\frac{1}{m}\right)G_{m,\chi}\left(\frac{1}{m}\right) \neq 0$.
\item Suppose that $d \neq 2$. If $d$ and $m$ are coprime, then $\mathcal{U}_0 = \{1\}$. If $d$ is prime and $m$ is a multiple of $d$, then $\mathcal{U}_0 = \{\chi \in \mathcal{U}: \chi^d = 1\}$.
\end{enumerate}
\end{lemma}
\begin{proof} \leavevmode
From Theorem \ref{hlfthm} and Corollary \ref{gftcor}, $\chi \in \mathcal{U}$ contributes to the pole of $Z_m\left(s\right)$ at $s = \frac{1}{m}$ if and only if each factor of $F_{m,\chi}\left(s\right)$ in \eqref{regeq} equals $\zeta_E\left(ms\right)$ and $\prod_{v \mid \infty}\widehat{H}_v\left(1,\chi_v;-\frac{1}{m}\right)G_{m,\chi}\left(\frac{1}{m}\right) \neq 0$. Denoting by $\psi$ the Hecke character associated to $\chi$, this means precisely that $\prod_{v \mid \infty}\widehat{H}_v\left(1,\chi_v;-\frac{1}{m}\right)G_{m,\chi}\left(\frac{1}{m}\right) \neq 0$ and, for each $\overline{\left(g_1,\dots,g_m\right)} \in S'\left(G,m\right)$, we have $\left(\psi^{g_1}\dots\psi^{g_m}\right)_v = 1$ for all $v \nmid \infty$, which is equivalent by strong approximation \cite[Thm.,~p.~67]{ANTCF} to $\psi^{g_1}\dots\psi^{g_m}= 1$. By Note \ref{torhecnote}, this holds if and only if
\begin{equation} \label{conteq}
\chi^{g_1}\dots\chi^{g_m} = 1 \textrm{ for all $\overline{\left(g_1,\dots,g_m\right)} \in S'\left(G,m\right)$.}
\end{equation}
To conclude the first part, it only remains to show that \eqref{conteq} holds if and only if $\chi \in \mathcal{U}_0$. Taking $\overline{\left(g_1,\dots,g_m\right)} = \overline{\left(1,\dots,1\right)}$ in \eqref{conteq}, we obtain $\chi^m =1$. If $d=2$, then $S'\left(G,m\right) =  \left\{\overline{\left(1,\dots,1\right)}\right\}$, and we are done. Otherwise, taking $\overline{\left(g_1,\dots,g_m\right)} = \overline{\left(g,1,\dots,1\right)}$ for any $g \in G$, we obtain $\chi^{m-1}\chi^g = 1$, so $\chi^m = 1$ and $\chi =\chi^g$ for all $g \in G$. Conversely, if $\chi^m = 1$ and $\chi = \chi^g$ for all $g \in G$, then \eqref{conteq} holds.

Let now $d \neq 2$, $\chi \in \mathcal{U}_0$, $v \not\in S'\left(\om\right)$ and $w \mid v$. We have $\psi_w\left(\pi_w\right) = \psi_w^g\left(\pi_w\right) = \psi_{gw}\left(\pi_{gw}\right)$ for all $g \in G$. Since $\prod_{w \mid v} \psi_w\left(\pi_w\right) = 1$ (see Note \ref{torhecnote}) and $G$ acts transitively on the places of $E$ over $v$, we obtain $\psi_w^d\left(\pi_w\right) = 1$, hence $\chi^d = 1$ by strong approximation. On the other hand, $\chi^m = 1$. For $d$ and $m$ coprime, we conclude that $\chi = 1$. \qedhere
\end{proof}
\begin{proposition} \label{lcprop}
The limit
$$
\Omega_m\left(\frac{1}{m}\right) = \lim_{s \rightarrow \frac{1}{m}}\left(s-\frac{1}{m}\right)^{b\left(d,m\right)}\sum_{\chi \in \mathcal{U}_0}\widehat{H}\left(\phi_m,\chi;-s\right)
$$
is non-zero.
\end{proposition}
\begin{proof}
We have
$$
\sum_{\chi \in \mathcal{U}_0}\widehat{H}\left(\phi_m,\chi;-s\right) = \sum_{\chi \in \mathcal{U}_0}\int_{T\left(\mathbb{A}_K\right)}\frac{\phi_m\left(t\right)\chi\left(t\right)}{H\left(t\right)^s}d\mu = \int_{T\left(\mathbb{A}_K\right)}\frac{\phi_m\left(t\right)}{H\left(t\right)^s} \sum_{\chi \in \mathcal{U}_0} \chi\left(t\right) d\mu.
$$
Let $t \in T\left(\mathbb{A}_K\right)$. Note that, if there exists $\chi' \in \mathcal{U}_0$ with $\chi'\left(t\right) \neq 1$, then
$$
\sum_{\chi \in \mathcal{U}_0}\chi\left(t\right) = \sum_{\chi \in \mathcal{U}_0} \chi\chi'\left(t\right) = \chi'\left(t\right)\sum_{\chi \in \mathcal{U}_0}\chi\left(t\right),
$$
so $\sum_{\chi \in \mathcal{U}_0}\chi\left(t\right) = 0$.
Then we have
$$
\begin{aligned}
\sum_{\chi \in \mathcal{U}_0}\widehat{H}\left(\phi_m,\chi;-s\right) = |\mathcal{U}_0|\int_{T\left(\mathbb{A}_K\right)^{\mathcal{U}_0,\phi_m}}\frac{1}{H\left(t\right)^s}d\mu,
\end{aligned}
$$
where
$$
T\left(\mathbb{A}_K\right)^{\mathcal{U}_0, \phi_m} = \{t \in T\left(\mathbb{A}_K\right) : \phi_m\left(t\right) = \chi\left(t\right) = 1 \textrm { for all } \chi \in \mathcal{U}_0\}.
$$
For any $\chi \in \mathcal{U}_0$ and non-archimedean place $v \not\in S'\left(\om\right)$, comparing the series expressions of $\widehat{H}_v\left(\phi_{m,v},\chi_v;-s\right)$ and $F_{m,\chi,v}\left(s\right) = F_{m,1,v}\left(s\right)$ in Definition \ref{pseriesdef}, we see that
$$
\int_{H_v\left(t_v\right) = 1}\chi_v\left(t_v\right)d\mu_v = \int_{H_v\left(t_v\right) = 1} d\mu_v, \quad
\int_{H_v\left(t_v\right) = q_v^m}\chi_v\left(t_v\right)d\mu_v = \int_{H_v\left(t_v\right) = q_v^m} d\mu_v,
$$
so $\chi_v\left(t_v\right) = 1$ for all $\chi \in \mathcal{U}_0$ whenever $H_v\left(t_v\right) = 1$ or $H_v\left(t_v\right) = q_v^m$.

For each place $v \not\in S'\left(\om\right)$, define the continuous function
$$
\theta_{m,v} : T\left(K_v\right) \rightarrow \{0,1\}, \quad
t_v \mapsto
\begin{cases}
1 \textrm{ if } H'_v\left(t_v\right) = 1 \textrm{ or } H'_v\left(t_v\right) = q_v^m, \\
0 \textrm { otherwise}. 
\end{cases}
$$
Letting $\theta_{m,v}$ be the indicator function of $\mathcal{K}_v$ for $v \in S'\left(\om\right)$, we define the function
$$
\theta_m : T\left(\mathbb{A}_K\right) \rightarrow \{0,1\}, \quad
\theta_m\left(\left(t_v\right)_v\right) = \prod_{v \in \Val\left(K\right)} \theta_{m,v}\left(t_v\right).
$$
By the above, we deduce that $T\left(\mathbb{A}_K\right)^{\theta_m} \subset T\left(\mathbb{A}_K\right)^{\mathcal{U}_0, \phi_m}$, where
$$
T\left(\mathbb{A}_K\right)^{\theta_m} = \{t \in T\left(\mathbb{A}_K\right) : \theta_m\left(t\right) = 1\}.
$$
Then, by comparing limits along the real line, we see that it suffices to prove that
$$
\lim_{s \rightarrow \frac{1}{m}}\left(s-\frac{1}{m}\right)^{b\left(d,m\right)}\widehat{H}\left(\theta_m,1;-s\right) \neq 0.
$$
It is easily seen that for any non-archimedean place $v \not\in S'\left(\om\right)$, we have
$$
\widehat{H}_v\left(\theta_{m,v},1;-s\right) = 1 + \frac{a_{1,v,m}}{q_v^{ms}}
$$
for $a_{\chi,v,n}$ as in Definition \ref{pseriesdef}, so, as in Corollary \ref{gftcor}, we may deduce that
$$
\widehat{H}\left(\theta_m,1;-s\right) = \zeta_E\left(ms\right)^{b\left(d,m\right)}G_m\left(s\right)
$$
for $G_m\left(s\right)$ a function holomorphic on $\Re s \geq \frac{1}{m}$. It also follows that $G_m\left(\frac{1}{m}\right) \neq 0$, since $\widehat{H}_v\left(\theta_{m,v},1;-\frac{1}{m}\right) \neq 0$ analogously to Lemma \ref{nontrivlem}. Then the result follows.
\end{proof}
\begin{corollary}
We have
$$
\Omega_m\left(\frac{1}{m}\right) = \frac{\Res_{s=1}\zeta_K\left(s\right)}{d\Res_{s=1}\zeta_E\left(s\right)} \lim_{s \rightarrow \frac{1}{m}}\left(s-\frac{1}{m}\right)^{b\left(d,m\right)}\sum_{\chi \in \mathcal{U}_0}\widehat{H}\left(\phi_m,\chi;-s\right) \neq 0.
$$
\end{corollary}
\begin{proof} [Proof of Theorem \ref{mainthm}]
Since $\Omega_m\left(\frac{1}{m}\right) \neq 0$, the result for $S = S\left(\om\right)$ follows from \cite[Thm.~3.3.2,~p.~624]{RPBH} and Theorem \ref{zetathm}, taking $c\left(\om,m,S\left(\om\right)\right)$ to be
\[
\frac{m\Res_{s=1}\zeta_K\left(s\right)}{\left(b\left(d,m\right)-1\right)!d \Res_{s=1}\zeta_E\left(s\right)}\lim_{s \rightarrow \frac{1}{m}}\left(s-\frac{1}{m}\right)^{b\left(d,m\right)}\sum_{\chi \in \mathcal{U}_0}\widehat{H}\left(\phi_m,\chi;-s\right). 
\]
The result for $S \supset S\left(\om\right)$ follows analogously upon redefining $\phi_{m,v}$ to be identically $1$ for each $v \in S \setminus S\left(\om\right)$ in Definition \ref{inddef}.
\end{proof}
%%
%%%%%%%%%%%%%%%%%%%%%%%%%%%%%%%%%%%%%%%%%%%%%%%%%%%%%%%%%%%%%%%%%%%%%%%%%%%%%%%%%%%%%%%%%%%%%%%%%%%%%%%%%%%%%%%%%%%%%%%%%%%%%%%%
\section{Campana points} \label{sectionCP}
%%%%%%%%%%%%%%%%%%%%%%%%%%%%%%%%%%%%%%%%%%%%%%%%%%%%%%%%%%%%%%%%
In this section we prove Theorem \ref{campthm}. We will be brief when the argument is largely similar to the case of weak Campana points, emphasising only the key differences. Fix a Galois extension $E/K$ of number fields with $K$-basis $\om = \{\omega_0,\dots,\omega_{d-1}\}$, let $m \in \mathbb{Z}_{\geq 2}$ and set $T = T_{\om}$ as in Section \ref{sectionTNT}. 
\begin{mydef} \label{campdef}
For each non-archimedean place $v \not\in S\left(\om\right)$, let
$$
N_{\om}\left(x\right) = \prod_{w \mid v} f_w\left(x\right)
$$
denote the $v$-adic decomposition of the norm form $N_{\om}$ associated to $\om$ into irreducible polynomials $f_w\left(x\right) \in \mathcal{O}_v[x]$.
For each $w \mid v$, we define the functions
$$
\widetilde{H}_w: T\left(K_v\right) \rightarrow \mathbb{R}_{>0}, \quad
x \mapsto \frac{\max\{|x_i|_v^{\degree f_w}\}}{|f_w\left(x\right)|_v},
$$
$$
\psi_{m,w}: T\left(K_v\right) \rightarrow \{0,1\}, \quad
t_v \mapsto
\begin{cases}
1 \textrm{ if } \widetilde{H}_w\left(t_v\right) = 1 \textrm{ or } \widetilde{H}_w\left(t_v\right) \geq q_v^m, \\
0 \textrm { otherwise}. 
\end{cases}
$$
We then define the \emph{Campana local indicator function}
$$
\psi_{m,v}: T\left(K_v\right) \rightarrow \{0,1\}, \quad
t_v \mapsto \prod_{w \mid v}\psi_{m,w}\left(t_v\right).
$$
Setting $\psi_{m,v} = 1$ for $v \in S\left(\om\right)$, we then define the \emph{Campana indicator function}
$$
\psi_m : T\left(\mathbb{A}_K\right) \rightarrow \{0,1\}, \quad
\left(t_v\right)_v \mapsto \prod_{v \in \Val\left(K\right)} \psi_{m,v}\left(t_v\right).
$$
If $v \not\in S'\left(\om\right)$, then for each $w \mid v$, we also define the function
$$
\sigma_{m,w}: T\left(K_v\right) \rightarrow \{0,1\},  \quad
t_v \mapsto
\begin{cases}
1 \textrm{ if } \widetilde{H}_w\left(t_v\right) = 1 \textrm{ or } \widetilde{H}_w\left(t_v\right) = q_v^m, \\
0 \textrm { otherwise}
\end{cases}
$$
and we define
$$
\sigma_{m,v}: T\left(K_v\right) \rightarrow \{0,1\}, \quad
t_v \mapsto \prod_{w \mid v}\sigma_{m,w}\left(t_v\right).
$$
Letting $\sigma_{m,v}$ be the indicator function for $\mathcal{K}_v$ for $v \in S'\left(\om\right)$, we define the function
$$
\sigma_m : T\left(\mathbb{A}_K\right) \rightarrow \{0,1\}, \quad
\left(t_v\right)_v \mapsto \prod_{v \in \Val\left(K\right)} \sigma_{m,v}\left(t_v\right).
$$
\end{mydef}
\begin{lemma}
The Campana $\mathcal{O}_{K,S\left(\om\right)}$-points of $\left(\mathbb{P}^{d-1}_K,\Delta^{\om}_m\right)$ are precisely the rational points $t \in T\left(K\right)$ such that $\psi_m\left(t\right) = 1$.
\end{lemma}
\begin{proof}
Taking coordinates $t_0,\dots,t_{d-1}$ as in the proof of Lemma \ref{identlem}, we have
$$
\widetilde{H}_w\left(t\right) = \frac{1}{|f_w\left(t_0,\dots,t_{d-1}\right)|_v}= q_v^{v\left(f_w\left(t_0,\dots,t_{d-1}\right)\right)} = q_v^{n_{\alpha_w}\left(\mathcal{Z}\left(f_w\right),t\right)}
$$
for all non-archimedean places $v \not\in S\left(\om\right)$ and places $w \mid v$, where $\mathcal{Z}\left(f_w\right)$ denotes the Zariski closure of $Z\left(f_w\right)$ in $\mathbb{P}^{d-1}_{\mathcal{O}_{K,S\left(\om\right)}}$.
\end{proof}
\begin{lemma}
For all $v \in \Val\left(K\right)$, the function $\psi_{m,v}$ is $\mathcal{K}_v$-invariant and $1$ on $\mathcal{K}_v$.
\end{lemma}
\begin{proof}
For $v \in S\left(\om\right)$ the result is trivial, so let $v \not\in S\left(\om\right)$ and $w \mid v$. Since $f_w\left(x \cdot y\right) = f_w\left(x\right) f_w\left(y\right)$ for $x,y \in L$, it follows as in the proof of Lemma \ref{smlem} that $\widetilde{H}_w\left(x \cdot y\right) \leq \widetilde{H}_w\left(x\right)\widetilde{H}_w\left(y\right)$ for all $x,y \in T\left(K_v\right)$. Since $H_v' = \prod_{w \mid v}\widetilde{H}_w$, we have $\widetilde{H}_w\left(\mathcal{K}_v\right) = 1$ for all $w \mid v$, hence it follows as in the proof of Corollary \ref{invcor} that $\widetilde{H}_w$ and $\psi_{m,w}$ are $\mathcal{K}_v$-invariant. Since $\psi_{m,v}\left(1\right) = 1$, we deduce that $\psi_{m,v}\left(\mathcal{K}_v\right) = 1$.
\end{proof}
\begin{proposition} \label{btheightprop}
Given $v \not\in S'\left(\om\right)$ and $t_v \in T\left(K_v\right)$, the image of $t_v$ in $X_*\left(T_v\right)$ is 
$$
\sum_{w \mid v} \frac{\log_{q_v}\left(\widetilde{H}_w\left(t_v\right)\right)}{\deg f_w}n_w,
$$
with $n_w$ defined as in Proposition \ref{cocharprop}.
\end{proposition}
\begin{proof}
Follows from Proposition \ref{cocharprop}.
\end{proof}
\begin{corollary} \label{lftcampcor}
For $v \not\in S'\left(\om\right)$ non-archimedean with $q_v$ sufficiently large, $\chi$ an automorphic character of $T$ unramified at $v$ and $s \in \mathbb{C}$ with $\Re s > 0$, we have
$$
\widehat{H}_v \left(\psi_{m,v},\chi_v;-s\right) = 1 + \frac{1}{q_v^{ms}}\sum_{\substack{w \mid v \\ \deg f_w \mid m}} \chi_w\left(\pi_w\right)^m + O\left(\frac{1}{q_v^{\left(m+1\right)s}}\right).
$$
\end{corollary}
\begin{proof}
Since $\chi_v$, $H_v$ and $\psi_{m,v}$ are $T\left(\mathcal{O}_v\right)$-invariant and $v \not\in S'\left(\om\right)$, we have
$$
\begin{aligned}
\widehat{H}_v \left(\psi_{m,v},\chi_v;-s\right) = \int_{T\left(K_v\right)}\frac{\psi_{m,v}\left(t_v\right)\chi_v\left(t_v\right)}{H_v\left(t_v\right)^s}d\mu_v = \sum_{\overline{t_v} \in T\left(K_v\right)/T\left(\mathcal{O}_v\right)}\frac{\psi_{m,v}\left(\overline{t_v}\right)\chi_v\left(\overline{t_v}\right)}{H_v\left(\overline{t_v}\right)^s} & \\
= \sum_{n_v \in X_*\left(T_v\right)} \frac{\psi_{m,v}\left(n_v\right)\chi_v\left(n_v\right)}{e^{\varphi_{\Sigma}\left(n_v\right) s\log q_v}} = \sum_{r=0}^{\infty} \frac{\gamma_{\chi,v,r}}{q_v^{rs}} &,
\end{aligned}
$$
where
$$
\gamma_{\chi,v,r} = \sum_{\substack{n_v \in X_*\left(T_v\right) \\ H_v\left(n_v\right) = q_v^r}} \psi_{m,v}\left(n_v\right)\chi_v\left(n_v\right).
$$
Put $d_w = \degree f_w$ and let $n_v = \sum_{w \mid v} \alpha_w n_w \in X_*\left(T_v\right)$ with $\min_w\{\alpha_w\} = 0$. By Proposition \ref{cocharprop} and Note \ref{torhecnote}, we have
\[
\begin{aligned}
\log_{q_v} H_v\left(n_v\right) & = \sum_{w \mid v} d_w \alpha_w, \quad
\chi_v\left(n_v\right) = \prod_{w \mid v}\chi_w\left(\pi_w\right)^{d_w \alpha_w}, \\
\psi_{m,v}\left(n_v\right) & =
\begin{cases}
1 \textrm{ if } \alpha_w = 0 \textrm { or } \alpha_w \geq \frac{m}{d_w} \textrm { for all } w \mid v, \\
0 \textrm { otherwise.}
\end{cases}
\end{aligned}
\]
In particular, $\psi_{m,v}\left(n_v\right) = 0$ whenever $q_v \leq H_v\left(n_v\right) \leq q_v^{m-1}$, hence $\gamma_{\chi,v,r} = 0$ for $1 \leq r \leq m-1$. Further, we see that $\psi_{m,v}\left(n_v\right) =1$ and $H_v\left(n_v\right) = q_v^m$ if and only if there is exactly one place $w_0 \mid v$ such that $\alpha_{w_0} = \frac{m}{d_{w_0}}$ and $\alpha_w = 0$ for $w \neq w_0$, so
$$
\gamma_{\chi,v,m} = \sum_{\substack{w \mid v \\ \deg f_w \mid m}} \chi_w\left(\pi_w\right)^m.
$$
Since $|\psi_{m,v}\left(n_v\right)\chi_v\left(n_v\right)| \leq 1$, we deduce that
$$
|\gamma_{\chi,v,r}| \leq \#\left\{\beta_1, \dots, \beta_d \in \mathbb{Z}_{\geq 0} : \sum_{i=1}^d \beta_i = r\right\} \leq d^r.
$$
Analogously to the proof of Corollary \ref{gftcor}, we deduce for $q_v$ sufficiently large that
\[
\sum_{r=m+1}^\infty\frac{\gamma_{\chi,v,r}}{q_v^{rs}} = O\left(\frac{1}{q_v^{\left(m+1\right)s}}\right). \qedhere
\]
\end{proof}
\begin{proposition} \label{lftcampprop}
For all places $v \not\in S'\left(\om\right)$ with $q_v$ sufficiently large, we have
$$
\widehat{H}_v \left(\psi_{m,v},\chi_v;-s\right) = L_v\left(\chi^m,ms\right)\left(1+ O\left(\frac{1}{q_v^{\left(m+1\right)s}}\right)\right), \quad \Re s > 0.
$$
\end{proposition}
\begin{proof}
Let $v \not\in S'\left(\om\right)$ with $q_v$ sufficiently large as in Corollary \ref{lftcampcor}. If $v$ is totally split in $E/K$, then $\deg f_w = 1$ for all $w \mid v$, so Corollary \ref{lftcampcor} gives
$$
\widehat{H}_v\left(\psi_{m,v},\chi_v;-s\right) = 1 + \frac{1}{q_v^{ms}}\sum_{w \mid v} \chi_w^m\left(\pi_w\right) + O\left(\frac{1}{q_v^{\left(m+1\right)s}}\right),
$$
and so we deduce the equality, since
$$
L_v\left(\chi^m,ms\right) = \prod_{w \mid v}\left(1-\frac{\chi_w^m\left(\pi_w\right)}{q_v^{ms}}\right)^{-1} = 1 + \frac{1}{q_v^{ms}}\sum_{w \mid v} \chi_w^m\left(\pi_w\right) + O\left(\frac{1}{q_v^{\left(m+1\right)s}}\right).
$$

Now let $v$ have inertia degree $d_v > 1$ in $E/K$. Then $\degree f_w = d_v \mid d$ for all $w \mid v$. If $d$ and $m$ are coprime, then $d_v \nmid m$, hence $\gamma_{\chi,v,m} = 0$ and the result follows from
$$
L_v\left(\chi^m,ms\right) = \prod_{w \mid v}\left(1-\frac{\chi^m_w\left(\pi_w\right)}{q_v^{d_v ms}}\right)^{-1} = 1 + O\left(\frac{1}{q_v^{d_v ms}}\right).
$$
If $d$ is prime, then $v$ is inert, so $T\left(\mathcal{O}_v\right) = T\left(K_v\right)$, $\widehat{H}_v\left(\psi_{m,v},\chi_v;-s\right) = 1$, and
\[
L_v\left(\chi^m,ms\right) = 1-\frac{1}{q_v^{dms}} = 1 + O\left(\frac{1}{q_v^{\left(m+1\right)s}}\right). \qedhere
\]
\end{proof}
\begin{proposition} \label{gftcampprop}
For any $\chi \in \mathcal{U}$, we have
$$
\widehat{H} \left(\psi_m,\chi;-s\right) = \prod_{v \mid \infty} \widehat{H}_v\left(1,\chi_v;-s\right)L\left(\chi^m,ms\right)\widetilde{G}_{m,\chi}\left(s\right),
$$
where $\widetilde{G}_{m,\chi}\left(s\right)$ is a function which is holomorphic on $\Re s \geq \frac{1}{m}$, $\widetilde{G}_{m,1}\left(\frac{1}{m}\right) \neq 0$ and $\widehat{H} \left(\psi_m,1;-s\right)$ has a simple pole at $s = \frac{1}{m}$.
\end{proposition}
\begin{proof}
Defining $\widetilde{G}_{m,\chi,v}\left(s\right) = \frac{\widehat{H}_v \left(\psi_{m,v},\chi_v;-s\right)}{L_v\left(\chi^m,ms\right)}$ for each place $v \nmid \infty$, it follows as in the proof of Corollary \ref{gftcor} that $\widetilde{G}_{m,\chi,v}\left(s\right)$ is holomorphic and bounded uniformly in terms of $\epsilon$ and $v$ on $\Re s \geq \epsilon$ for all $\epsilon > 0$. Since Proposition \ref{lftcampprop} gives
$$
\widetilde{G}_{m,\chi,v}\left(s\right) = 1 + O\left(\frac{1}{q_v^{\left(m+1\right)s}}\right),
$$ 
it follows as in the proof of Corollary \ref{gftcor} that $\widetilde{G}_{m,\chi}\left(s\right)$ is holomorphic and uniformly bounded with respect to $\chi$ for $\Re s \geq \frac{1}{m}$ with $\widetilde{G}_{m,1}\left(\frac{1}{m}\right) \neq 0$. Then, since
$$
L\left(1,ms\right) = \zeta_E\left(ms\right),
$$
we conclude from Theorem \ref{hlfthm} that $\widehat{H} \left(\psi_m,1;-s\right)$ has a simple pole at $s = \frac{1}{m}$.
\end{proof}
%%
%%\
\begin{mydef}
For $\Re s \gg 0 $, define the functions
\[
\widetilde{Z}_m : \mathbb{C} \rightarrow \mathbb{C}, \quad
s \mapsto \sum_{x \in \mathbb{P}^{d-1}\left(K\right)}\frac{\psi_m\left(x\right)}{H\left(x\right)^s}, \quad \widetilde{\Omega}_m = \widetilde{Z}_m\left(s\right)\left(s-\frac{1}{m}\right).
\]
\end{mydef}
The proofs of the following two results are analogous to those of their weak Campana counterparts, namely Theorem \ref{zetathm} and Lemma \ref{contriblem} respectively. 
\begin{theorem} \label{zetatildthm}
The function $\widetilde{\Omega}_m\left(s\right)$ admits a holomorphic extension to $\Re s \geq \frac{1}{m}$.
\end{theorem}
\begin{lemma}
The characters $\chi \in \mathcal{U}$ contributing to the simple pole of $\widetilde{Z}_m\left(s\right)$ at $s = \frac{1}{m}$ are precisely the characters $\chi \in \mathcal{U}[m]$ such that $\widetilde{G}_{m,\chi}\left(\frac{1}{m}\right) \neq 0$.
\end{lemma}
\begin{proposition} \label{omtildprop}
The limit
$$
\lim_{s \rightarrow \frac{1}{m}}\left(s-\frac{1}{m}\right)\sum_{\chi \in \mathcal{U}[m]}\widehat{H}\left(\psi_m,\chi;-s\right)
$$
is non-zero.
\end{proposition}
\begin{proof}
By the same reasoning as in the proof of Proposition \ref{lcprop}, we have
$$
\sum_{\chi \in \mathcal{U}[m]}\widehat{H}\left(\psi_m,\chi;-s\right) = \sum_{\chi \in \mathcal{U}[m]}\int_{T\left(\mathbb{A}_K\right)}\frac{\psi_m\left(t\right)\chi\left(t\right)}{H\left(t\right)^s}d\mu = |\mathcal{U}[m]|\int_{T\left(\mathbb{A}_K\right)^{\mathcal{U}[m],\psi_m}}\frac{1}{H\left(t\right)^s}d\mu,
$$
where
$$
T\left(\mathbb{A}_K\right)^{\mathcal{U}[m], \psi_m} = \{t \in T\left(\mathbb{A}_K\right) : \psi_m\left(t\right) = \chi\left(t\right) = 1 \textrm { for all } \chi \in \mathcal{U}[m]\}.
$$
Now, take $\chi \in \mathcal{U}[m]$, $v \not\in S'\left(\om\right)$ non-archimedean. If $\sigma_v\left(t_v\right) = 1$ for some $t_v \in T\left(K_v\right)$, then the image of $t_v$ in $X_*\left(T_v\right)$ is of the form $\sum_{w \mid v} \alpha_w n_w$, where each $\alpha_w$ is either $0$ or $\frac{m}{d_v}$ for $d_v$ the common inertia degree of the places of $E$ over $v$, so
$$
\chi_v\left(t_v\right) = \prod_{w \mid v}\chi_w\left(\pi_w\right)^{d_v\alpha_w} = 1,
$$
since each $d_v \alpha_w$ is $0$ or $m$ and $\chi^0 = \chi^m = 1$.
Then $\chi_v\left(t_v\right) = 1$ for all $\chi \in \mathcal{U}[m]$.
In particular, we deduce that $T\left(\mathbb{A}_K\right)^{\sigma_m} \subset T\left(\mathbb{A}_K\right)^{\mathcal{U}[m], \psi_m}$, where
$$
T\left(\mathbb{A}_K\right)^{\sigma_m} = \{t \in T\left(\mathbb{A}_K\right) : \sigma_m\left(t\right) = 1\}.
$$
Then it suffices to prove that
$$
\lim_{s \rightarrow \frac{1}{m}}\left(s-\frac{1}{m}\right)\widehat{H}\left(\sigma_m,1;-s\right) \neq 0.
$$
Analogously to the proof of Proposition \ref{gftcampprop}, we may deduce that
$$
\widehat{H}\left(\sigma_m,1;-s\right) = \zeta_E\left(ms\right) \widetilde{G}_m\left(s\right)
$$
for $\widetilde{G}_m\left(s\right)$ a function holomorphic on $\Re s \geq \frac{1}{m}$ with $\widetilde{G}_m\left(\frac{1}{m}\right) \neq 0$, so the result follows.
\end{proof}
\begin{proof} [Proof of Theorem \ref{campthm}]
Since $\widetilde{\Omega}_m\left(\frac{1}{m}\right) \neq 0$ by Proposition  \ref{omtildprop}, the result for $S = S\left(\om\right)$ now follows from \cite[Thm.~3.3.2,~p.~624]{RPBH} and Theorem \ref{zetatildthm}, taking
\[
\widetilde{c}\left(\om,m,S\left(\om\right)\right) = \frac{m\Res_{s=1}\zeta_K\left(s\right)}{d\Res_{s=1}\zeta_E\left(s\right)} \lim_{s \rightarrow \frac{1}{m}} \left(s-\frac{1}{m}\right)\sum_{\chi \in \mathcal{U}[m]}\widehat{H}\left(\psi_m,\chi;-s\right).
\]
The result for $S \supset S\left(\om\right)$ follows analogously upon redefining $\psi_{m,v}$ to be identically $1$ for each $v \in S \setminus S\left(\om\right)$.
\end{proof}
%%
%%%%%%%%%%%%%%%%%%%%%%%%%%%%%%%%%%%%%%%%%%%%%%%%%%%%%%%%%%%%%%%%
%%%%%%%%%%%%%%%%%%%%%%%%%%%%%%%%%%%%%%%%%%%%%%%%%%%%%%%%%%%%%%%%
\section{Comparison to Manin-type conjecture}
%%%%%%%%%%%%%%%%%%%%%%%%%%%%%%%%%%%%%%%%%%%%%%%%%%%%%%%%%%%%%%%%
In this section we compare the leading constant in Theorem \ref{campthm} with the Manin--Peyre constant in the conjecture of Pieropan, Smeets, Tanimoto and V\'arilly-Alvarado.
%%%%%%%%%%%%%%%%%%%%%%%%%%%%%%%%%%%%%%%%%%%%%%%%%%%%%%%%%%%%%%%%
\subsection{Statement of the conjecture}
%%%%%%%%%%%%%%%%%%%%%%%%%%%%%%%%%%%%%%%%%%%%%%%%%%%%%%%%%%%%%%%%
Let $\left(X,D_{\epsilon}\right)$ be a smooth Campana orbifold over a number field $K$ which is \emph{klt} (i.e.\ $\epsilon_{\alpha} < 1$ for all $\alpha \in \mathcal{A}$) and \emph{Fano} (i.e.\ $-\left(K_X + D_{\epsilon}\right)$ is ample). Let $\left(\mathcal{X},\mathcal{D}_{\epsilon}\right)$ be a regular $\mathcal{O}_{K,S}$-model of $\left(X,D_{\epsilon}\right)$ for some finite set $S \subset \Val\left(K\right)$ containing $S_{\infty}$ (i.e.\ $\mathcal{X}$ regular over $\mathcal{O}_{K,S}$). Let $\mathcal{L} = \left(L,\|\cdot\|\right)$ be an adelically metrised big line bundle with associated height function $H_{\mathcal{L}}: X\left(K\right) \rightarrow \mathbb{R}_{>0}$ (see \cite[\S1.3]{HMT}). For any subset $U \subset X\left(K\right)$ and any $B \in \mathbb{R}_{>0}$, we define
$$
N\left(U,\mathcal{L},B\right) = \#\{P \in U: H_{\mathcal{L}}\left(P\right) \leq B\}.
$$
\begin{mydef}
Let $V$ be a variety over a field $k$ of characteristic zero, and let $A \subset V\left(k\right)$.
We say that $A$ is of \emph{type I} if there is a proper Zariski closed subset $W \subset V$ with $A \subset W\left(k\right)$. We say that $A$ is of \emph{type II} if there is a normal geometrically irreducible variety $V'$ with $\textrm{dim}V' = \textrm{dim}V$ and a finite surjective morphism $\phi: V' \rightarrow V$ of degree $\geq 2$ with $A \subset \phi\left(V'\left(k\right)\right)$. We say that $A$ is \emph{thin} if it is contained in a finite union of subsets of $V\left(k\right)$ of types I and II.  
\end{mydef}
We are now ready to give the statement of the conjecture.
\begin{conjecture} \cite[Conj.~1.1,~p.~3]{CPBH} \label{manconj}
Suppose that $L$ is nef and $\left(\mathcal{X},\mathcal{D}_{\epsilon}\right)\left(\mathcal{O}_{K,S}\right)$ is not thin. Then there exists a thin set $Z \subset \left(\mathcal{X},\mathcal{D}_{\epsilon}\right)\left(\mathcal{O}_{K,S}\right)$ and explicit positive constants $a = a\left(\left(X,D_{\epsilon}\right),L\right)$, $b = b\left(K,\left(X,D_{\epsilon}\right),L\right)$ and $c = c\left(K,S,\left(\mathcal{X},\mathcal{D}_{\epsilon}\right),\mathcal{L},Z\right)$ such that, as $B \rightarrow \infty$, we have
$$
N\left(\left(\mathcal{X},\mathcal{D}_{\epsilon}\right)\left(\mathcal{O}_{K,S}\right) \setminus Z,\mathcal{L},B\right) \sim c B^a\left(\log B\right)^{b-1}.
$$
\end{conjecture}
%%%%%%%%%%%%%%%%%%%%%%%%%%%%%%%%%%%%%%%%%%%%%%%%%%%%%%%%%%%%%%%%
\subsection{Interpretation for norm orbifolds}
%%%%%%%%%%%%%%%%%%%%%%%%%%%%%%%%%%%%%%%%%%%%%%%%%%%%%%%%%%%%%%%%
The orbifold $\left(\mathbb{P}^{d-1}_K,\Delta^{\om}_m\right)$ in Theorem~\ref{campthm} is klt and Fano. It is smooth precisely when $d=2$. The $\mathcal{O}_{K,S\left(\om\right)}$-model $\left(\mathbb{P}^{d-1}_{\mathcal{O}_{K,S\left(\om\right)}},\mathcal{D}_m^{\om}\right)$ is regular. The Batyrev--Tschinkel height arises from an adelic metrisation $\mathcal{L}$ of $-K_{\mathbb{P}^{d-1}} = \mathcal{O}\left(d\right)$. According to \cite[\S3.3]{CPBH}, we have
$$
c\left(K,S,\left(\mathbb{P}^{d-1}_{\mathcal{O}_{K,S\left(\om\right)}},\mathcal{D}_m^{\om}\right),\mathcal{L},Z\right) = \frac{1}{d}\tau\left(K,S\left(\om\right),\left(\mathbb{P}^{d-1}_K,\Delta^{\om}_m\right),\mathcal{L},Z\right),
$$
where
$$
\tau\left(K,S\left(\om\right),\left(\mathbb{P}^{d-1}_K,\Delta^{\om}_m\right),\mathcal{L},Z\right) = \int_{\overline{\mathbb{P}^{d-1}\left(K\right)}_\epsilon}H\left(t\right)^{1-\frac{1}{m}}d\tau_{\mathbb{P}^{d-1}}.
$$
Here, $\tau_{\mathbb{P}^{d-1}}$ is the Tamagawa measure defined in \cite[Def.~2.8,~p.~372]{NVF}, and $\overline{\mathbb{P}^{d-1}\left(K\right)}_\epsilon$ denotes the topological closure of the Campana points inside $\mathbb{P}^{d-1}\left(\mathbb{A}_K\right)$. If one assumes that weak approximation for Campana points holds for this orbifold (see \cite[Question~3.9,~p.~13]{CPBH}), it follows from the definition of $\tau_{\mathbb{P}^{d-1}}$ (cf.\ \cite[\S9]{CPBH}) that 
$$
\tau = m\frac{\Res_{s=1}\zeta_K\left(s\right)}{\Res_{s=1}\zeta_E\left(s\right)}\lim_{s\rightarrow \frac{1}{m}}\left(s-\frac{1}{m}\right)\widehat{H}\left(\psi_m,1;-s\right).
$$
Given the assumption on weak approximation, the conjectural leading constant is
$$
c\left(K,S,\left(\mathbb{P}^{d-1}_{\mathcal{O}_{K,S\left(\om\right)}},\mathcal{D}_m^{\om}\right),\mathcal{L},Z\right)= \frac{m\Res_{s=1}\zeta_K\left(s\right)}{d\Res_{s=1}\zeta_E\left(s\right)}\lim_{s\rightarrow \frac{1}{m}}\left(s-\frac{1}{m}\right)\widehat{H}\left(\psi_m,1;-s\right).
$$
On the other hand, the leading constant given by Theorem \ref{campthm} in this case is
$$
\widetilde{c}\left(\om,m,S\left(\om\right)\right) = \frac{m\Res_{s=1}\zeta_K\left(s\right)}{d\Res_{s=1}\zeta_E\left(s\right)} \lim_{s \rightarrow \frac{1}{m}} \left(s-\frac{1}{m}\right)\sum_{\chi \in \mathcal{U}[m]}\widehat{H}\left(\psi_m,\chi;-s\right).
$$
We observe that our constant differs from the conjectural one in the potential inclusion of non-trivial characters in the limit.
%%%%%%%%%%%%%%%%%%%%%%%%%%%%%%%%%%%%%%%%%%%%%%%%%%%%%%%%%%%%%%%%
\subsection{The quadratic case}
%%%%%%%%%%%%%%%%%%%%%%%%%%%%%%%%%%%%%%%%%%%%%%%%%%%%%%%%%%%%%%%%
We now consider the case $d=2$, in which the orbifold in Theorem \ref{campthm} is smooth. Here, work of Nakahara and the author \cite{HPCP} shows that weak approximation for Campana points holds for any $m \in \mathbb{Z}_{\geq 2}$.

In Theorem \ref{campthm}, it is not clear that there are non-trivial characters contributing to the leading constant and whether their contribution is positive if so. However, we now exhibit an extension for which a non-trivial character contributes positively to the leading constant, and all contributing characters do so positively.
\begin{proposition} \label{campprop}
Let $K = \mathbb{Q}\left(\sqrt{-39}\right)$, $E = \mathbb{Q}\left(\sqrt{-3},\sqrt{13}\right)$ and $m=2$. Choose the $K$-basis $\om = \{1,\frac{1+\sqrt{-3}}{2}\}$ of $E$. Then $S\left(\om\right) = S_{\infty}$, $\#\mathcal{U}[2] > 1$, and for every $\chi \in \mathcal{U}[2]$, we have $\lim_{s\rightarrow\frac{1}{2}}\left(s-\frac{1}{2}\right)\widehat{H}\left(\psi_2,\chi;-s\right) > 0$. 
\end{proposition}
\begin{proof}
Writing $a \cdot 1 + b \cdot \frac{1 + \sqrt{-3}}{2}$ as $\left(a,b\right)$ and $G = \Gal\left(E/K\right)$ as $\{1_G,g\}$, we have
$$
\left(1,0\right)^2 = \left(1,0\right), \, \left(1,0\right) \cdot \left(0,1\right) = \left(0,1\right), \, \left(0,1\right)^2 = \left(-1,1\right),
$$
$$
1_G\left(\left(1,0\right)\right) = \left(1,0\right), \, 1_G\left(\left(0,1\right)\right) = \left(0,1\right), \, g\left(\left(1,0\right)\right) = \left(1,0\right), \, g\left(\left(0,1\right)\right) = \left(1,-1\right),
$$
and clearly $1 = \left(1,0\right)$, hence $S\left(\om\right) = S_{\infty}$. Note that $N_{\om}\left(x,y\right) = x^2 + xy + y^2$.

Since $\Cl\left(E\right) \cong \mathbb{Z}/2\mathbb{Z}$, the Hilbert class field $M$ of $E$ is quadratic over $E$. We obtain the unramified Hecke character $\chi_M$ of $E$, which is defined for all split $w \in \Val\left(E\right)$ by $\chi_{M,w}\left(\pi_w\right) = -1$ and is trivial at all other places. Since $\chi_M$ is trivial on $\mathbb{A}_K^*$, it may be viewed inside $\mathcal{U}[2]$, hence $\#\mathcal{U}[2] > 1$. 

Let $\chi \in \mathcal{U}[2]$. To show that $\lim_{s\rightarrow\frac{1}{2}}\left(s-\frac{1}{2}\right)\widehat{H}\left(\psi_2,\chi;-s\right) > 0$, it suffices to show that $\widehat{H}_v\left(\psi_{2,v},\chi_v;-\frac{1}{2}\right) > 0$ for all $v \in \Val\left(K\right)$. If $v \mid \infty$, then $\widehat{H}_v\left(\psi_{2,v},\chi_v;-\frac{1}{2}\right) = \widehat{H}_v\left(1,1;-\frac{1}{2}\right)  > 0$, as $\chi_v$ gives a continuous homomorphism from $T\left(K_v\right)/T\left(\mathcal{O}_v\right) \cong \mathbb{R}_{>0}$ to $\mu_2$, and $\mathbb{R}_{>0}$ has no proper open subgroups. If $v$ is inert, then we have $\widehat{H}_v\left(\psi_{2,v},\chi_v;-\frac{1}{2}\right) = 1$. If $v$ is not inert, then $N_{\om}\left(x,y\right) = \left(x+\theta_1 y\right) \left(x+\theta_2 y\right)$ for $\theta_1,\theta_2 \in \mathcal{O}_v$ roots of $z^2 - z + 1$. By Proposition \ref{btheightprop}, we have $H_v = H_v'$ if and only if there are no $\left(a,b\right) \in \left(K_v^*\right)^2$ with $\min\{v\left(a\right),v\left(b\right)\} = 0$ such that $v\left(a+\theta_1 b\right), v\left(a + \theta_2 b\right) \geq 1$. If $v\left(a+\theta_1 b\right), v\left(a + \theta_2 b\right) \geq 1$, then we deduce from the equalities
$$
\theta_1\left(a + \theta_2 b\right) - \theta_2 \left(a + \theta_1 b\right) = \left(\theta_1 - \theta_2\right) a, \quad \left(a + \theta_1 b\right) - \left(a + \theta_2 b\right) = \left(\theta_1 - \theta_2\right)b,  
$$
that $v\left(a\right), v\left(b\right) \geq 1 - v\left(\theta_1 - \theta_2\right)$. Since $\min\{v\left(a\right),v\left(b\right)\} = 0$, we have $H_v' \neq H_v$ if and only if $v\left(\theta_1 - \theta_2\right) \geq 1$. Since $\left(\theta_1 - \theta_2\right)^2 = -3$, the only such place is the unique place $v_0$ of $K$ above $3$, and $v_0\left(\theta_1 - \theta_2\right) = 1$.

For any split place $v \neq v_0$, we have $\mathcal{K}_v = T\left(\mathcal{O}_v\right)$ and $\psi_{2,v} = \phi_{2,v}$, so
\[
\widehat{H}_v\left(\psi_{2,v},\chi_v;-\frac{1}{2}\right) = 1 + \sum_{n=2}^{\infty}\frac{c_{\chi,v,n} - c_{\chi,v,n-2}}{q_v^{\frac{n}{2}}}
\]
by Proposition \ref{truncprop}.
In fact, for $w_1$ and $w_2$ the places of $E$ over $v$, we have $\chi_{w_2}\left(\pi_{w_2}\right) = \chi_{w_1}\left(\pi_{w_1}\right)^{-1} \in \{1,-1\}$, hence $c_{\chi,v,n} - c_{\chi,v,n-2} = 2\chi_{w_1}\left(\pi_{w_1}\right)^n$, so
$$
\widehat{H}_v\left(\psi_{2,v},\chi_v;-\frac{1}{2}\right) = 1 + \sum_{n=2}^{\infty}\frac{2\chi_{w_1}\left(\pi_{w_1}\right)^n}{q_v^{\frac{n}{2}}} = 1 + \frac{2}{q_v} \left(\frac{1}{1-q_v^{-\frac{1}{2}}\chi_{w_1}\left(\pi_{w_1}\right)}\right) > 0.
$$

It only remains to check that $\widehat{H}_{v_0}\left(\psi_{2,{v_0}},\chi_{v_0};-\frac{1}{2}\right) > 0$. We will make use of the following property of valuations:
\begin{equation} \label{valpropeq}
v_0\left(\alpha + \beta\right) \geq \min\{v_0\left(\alpha\right),v_0\left(\beta\right)\}, \, \textrm{with equality if $v_0\left(\alpha\right) \neq v_0\left(\beta\right)$.}
\end{equation}
 Assume that, for $a,b \in \left(K_{v_0}^*\right)^2$ as above, we have $v_0\left(a+\theta_2 b\right) \geq 2$. We claim that $v_0\left(a+\theta_1 b\right) = 1$. First, we deduce from \eqref{valpropeq} that
\begin{equation*}
\begin{aligned}
2 \leq v_0\left(a + \theta_2 b\right) & = v_0\left(\left(a + \theta_1 b\right) + \left(\theta_1 - \theta_2\right)b\right) \\
 & \geq \min\{v_0\left(a + \theta_1 b\right), v_0\left(\left(\theta_1 - \theta_2\right) b\right)\},
\end{aligned}
\end{equation*}
with equality if $v_0\left(a + \theta_1 b\right) \neq v_0\left(\left(\theta_1 - \theta_2\right) b\right)$. Since $v_0\left(\left(2\theta_i - 1\right)^2\right) = v_0\left(-3\right) = 2$, we have $v_0\left(2\theta_i - 1 \right) = 1$, so $v_0 \left(\theta_i\right) = v_0\left(2\theta_i\right) = v_0\left(1\right) =0$ by \eqref{valpropeq}, hence $\min\{v_0\left(a\right),v_0\left(\theta_2 b\right)\} = \min\{v_0\left(a\right),v_0\left(b\right)\} = 0$.
Then, since $v_0\left(a+\theta_2 b\right) \geq 2$, it follows that $v_0\left(a\right) = v_0\left(\theta_2 b\right)$. We deduce that $v_0\left(a\right) = v_0\left(\theta_2 b\right) = v_0\left(b\right)$, so $v_0\left(a\right) = v_0\left(b\right) = \min\{v_0\left(a\right),v_0\left(b\right)\} = 0$. Since $v_0\left(\theta_1 - \theta_2\right) = 1$, we have $v_0\left(\left(\theta_1 - \theta_2\right)b\right) = 1$. It follows that $v_0\left(a + \theta_1 b\right) = 1$ by \eqref{valpropeq}.

We deduce that $\psi_{2,v_0}\left(t_{v_0}\right) = 1$ if and only if $t_{v_0} \in \mathcal{K}_{v_0}$, hence
$$
\widehat{H}_{v_0}\left(\psi_{2,{v_0}},\chi_{v_0};-\frac{1}{2}\right) = \int_{\mathcal{K}_{v_0}}d\mu_{v_0} > 0;
$$
positivity follows since $\mathcal{K}_v \subset T\left(\mathcal{O}_v\right)$ is of finite index for all $v \in \Val\left(K\right)$.
\end{proof}
%%%%%%%%%%%%%%%%%%%%%%%%%%%%%%%%%%%%%%%%%%%%%%%%%%%%%%%%%%%%%%%%
\subsection{Possible thin sets}
%%%%%%%%%%%%%%%%%%%%%%%%%%%%%%%%%%%%%%%%%%%%%%%%%%%%%%%%%%%%%%%%
Assuming the truth of Conjecture \ref{manconj}, the question arises of which thin set $Z$ should be removed in the setting of Proposition \ref{campprop}. Informally, its removal should remove the contribution of all non-trivial characters $\chi \in \mathcal{U}[2]$. One might therefore postulate that, for each non-trivial character $\chi \in \mathcal{U}[2]$, there is a finite morphism $\varphi_{\chi}: C_{\chi} \rightarrow \mathbb{P}^1_K$, where $C_{\chi}$ is a smooth projective curve, and $Z = \bigcup_{\chi \in \mathcal{U}[2]}\varphi_{\chi}\left(C_\chi\left(K\right)\right)$. By the height bounds in \cite[\S9.7]{LMWT}, we would have $C_\chi \cong \mathbb{P}^1_K$ and $\degree\left(\varphi_\chi\right) =2$, making the morphisms $\varphi_{\chi}$ degree-two endomorphisms of $\mathbb{P}^1_K$. However, it is not clear how one should construct such endomorphisms. This may be an interesting direction to pursue in future work.
%%%%%%%%%%%%%%%%%%%%%%%%%%%%%%%%%%%%%%%%%%%%%%%%%%%%%%%%%%%%%%%%%%%%%%%%%%%%%%%%%%%%%%%%%%%%%%%%%%%%%%%%%%%%%%%%%%%%%%%%%%%%%%%%

%%%%%%%%%%%%%%%%%%%%%%%%%%%%%%%%%%%%%%%%%%%%%%%%%%%%%%%%%%%%%%%%%%%%%%%%%%%%%%%
%%%%%%%%%%%%%%%%%%%%%%%%%%%%%%%%%%%%%%%%%%%%%%%%%%%%%%%%%%%%%%%%%%%%%%%%%%%%%%%
\end{document}